\documentclass{article}
\usepackage{amsmath,amsthm,amssymb,bbm}

\newcommand{\assign}{:=}

\newcommand{\tmop}[1]{\ensuremath{\operatorname{#1}}}

\newcommand{\tmstrong}[1]{\textbf{#1}}

\newtheorem{corollary}{Corollary}
\newtheorem{lemma}{Lemma}
\newtheorem{proposition}{Proposition}
\newtheorem{theorem}{Theorem}

\newcommand{\XXint}[3]{{\setbox}0=\text{\ensuremath{#1 #2 #3 \int}}
{\vcenter{\text{\ensuremath{#2 #3}}}}{\kern}-.5{\tmwd}0}

\newcommand{\opn}[2]{\newcommand{\1}{\}} {\opn}{\Rm{Rm}} {\opn}{\Ric{Ric}}
{\opn}{\Rc{Rc}} {\opn}{\Scal{Sc}} {\opn}{\Tr{Tr}} {\opn}{\Trac{Tr}}
{\opn}detdet {\opn}{\diam{diam}} {\opn}{\dist{dist}} {\opn}{\Im}Im
{\opn}{\div}div {\opn}{\Ker{Ker}} {\opn}expexp {\opn}{\Vol{Vol}}
{\opn}{\exph{exph}} {\opn}{\Herm{Herm}} {\opn}{\End{End}} {\opn}{\Hess{Hess}}
{\opn}{\Vol{Vol}}}

\newcommand{\R}{\mathbb{R}}

\newcommand{\I}{\mathbb{I}}

\newcommand{\contract}{{\kern}-1.5pt{\vrule} width6.0pt height0.4pt depth0pt
{\vrule} width0.4pt height4.0pt depth0pt}
\newcommand{\retract}{{\kern}-1.5pt{\vrule} width0.4pt height4.0pt depth0pt
{\vrule} width6.0pt height0.4pt depth0pt}
\newcommand{\Openbox}{{\leavevmode} {\text{{\hfil}{\vrule}
width{\boxrulethickness} {\vbox} to{\Openboxwidth{{\advance}{\Openboxwidth}
-2{\boxrulethickness} {\hrule} height {\boxrulethickness}
width{\Openboxwidth}{\vfil} {\hrule} height{\boxrulethickness}}}{\vrule}
width{\boxrulethickness}{\hfil} }}}

\begin{document}

\title{ The total second variation of Perelman's $\mathcal{W}$-functional }\author{\\
{\tmstrong{NEFTON PALI}}}\maketitle

\begin{abstract}
  We show a very simple and general total second variation formula for
  Perelman's $\mathcal{W}$-functional at arbitrary points in the space of
  Riemannian metrics. Moreover we perform a study of the
  properties of the variations of K\"ahler structures. 
We deduce a quite simple and
  general total second variation formula for Perelman's
  $\mathcal{W}$-functional with respect to such variations. 
In this case the main therm in the formula depends 
strongly on the variation of the complex structure. We discover also convexity of Perelman's
  $\mathcal{W}$-functional along particular variations over points with non-negative Bakry-Emery-Ricci tensor.
\end{abstract}

\section{Introduction}

Total second variation formulas for Perelman's $\mathcal{W}$-functional at a 
shrinking Ricci soliton point were obtained independently by 
Cao-Hamilton-Ilmanenen \cite{C-H-I}, Cao-Zhu \cite{Ca-Zhu} and Tian-Zhu \cite{Ti-Zhu}. 
The work of Cao-Zhu \cite{Ca-Zhu} is based on the previous work of Cao-Hamilton-Ilmanenen \cite{C-H-I}.
Important applications to the stability and the convergence 
of the K\"ahler-Ricci flow over Fano manifolds were 
given by Tian-Zhu \cite{Ti-Zhu} and by Tian-Zhang-Zhang-Zhu \cite{T-Z-Z-Z}.
The second variation formulas obtained by the previous authors are of
different nature from ours. We explain now our set-up and results.

Let $(X, g)$ be a compact oriented (for simplicity) Riemannian manifold of
dimension $m$ and let $f$ be a smooth real valued function over $X$. We remind
that Perelman's $\mathcal{W}$-functional is defined (up to a constant) by the
formula
\begin{eqnarray*}
  \mathcal{W} (g, f) & : = & \int_X \left( | \nabla_g f|^2_g \;\, +
  \;\, \tmop{Scal}_g \;\, + \;\, 2\, f
  \;\, -\;\, m \right) e^{- f} dV_g \;.
\end{eqnarray*}
Let also $\Omega >0$ be a smooth volume form over $X$. We remind that a remarkable result due to Perelman \cite{Per} asserts that the first
variation of the functional
\[ \mathcal{W}_{\Omega} (g) \hspace{0.75em} : = \hspace{0.75em} \mathcal{W}
   (g, \log (dV_g / \Omega)) \hspace{0.25em}, \]
at a point $g$ is given by the formula
\begin{eqnarray*}
  D_g  \,\mathcal{W}_{\Omega} \,(v) & = & \int_X \left\langle v\,,\, g \;\,
  - \;\, \tmop{Ric}_g (\Omega) \right\rangle_g \Omega
  \hspace{0.25em},
\end{eqnarray*}
where $\tmop{Ric}_g (\Omega)$ is the $\Omega$-Bakry-Emery-Ricci tensor. (In
the appendix \ref{apdx2} we give a straightforward proof of this formula based
on the variation formula for the $\Omega$-Bakry-Emery-Ricci tensor.) This
formula shows that Perelman's modified Ricci flow type equation is the
gradient flow of the functional $\mathcal{W}_{\Omega}$ with respect to the
metric 
$$
G \;\;\assign\;\; \int_X \left\langle \cdot, \cdot \right\rangle_g \Omega\;\,
$$
over the space of Riemannian metrics. We remind that the space of Riemannian
metrics equipped with this metric is a non-positively curved space with
geodesics given by the formula 
$$
g_t \;\;=\;\; g_0 \,e^{t \,g^{- 1}_0 \dot{g}_0}\;.
$$
Thus it
is natural to compute the total second variation of the functional
$\mathcal{W}_{\Omega}$ with respect to the metric $G$. In this paper we \
obtain the following quite simple and general second variation result at arbitrary points in
the space of Riemannian metrics.

\begin{theorem}
  \label{lm-IIvr-W}Let $(X, g)$ be a compact and orientable Riemannan manifold
  and let $\Omega > 0$ be a smooth volume form. Along any smooth curve
  $(g_t)_{t \in (- \varepsilon, \varepsilon)}$ of Riemannian metrics hold the
  second variation formula
\begin{eqnarray}  
    \label{rm-2vr-W} \nabla_G \,D\, \mathcal{W}_{\Omega}\, (g_t)\, ( \dot{g}_t,
    \dot{g}_t) &= &\int_X \left\langle \dot{g}_t \cdot
    \tmop{Ric}^{\ast}_{g_t} (\Omega)\,,\, \dot{g}_t \right\rangle_{g_t} \Omega\nonumber
\\\nonumber
\\
&+&
    \int_X \left[ \,\frac{1}{6} \hspace{0.25em} \big|
    \hat{\nabla}_{g_t}  \,\dot{g}_t \big|^2_{g_t} \;\, - \;\,
    \big| \nabla_{g_t} \,\dot{g}_t \big|^2_{g_t} \right] \Omega \hspace{0.25em},
\end{eqnarray}
  where $\tmop{Ric}^{\ast}_{g_t} (\Omega)$ denotes the endomorphism section
  associated to the $\Omega$-Bakry-Emery-Ricci tensor and $\hat{\nabla}_g$
  denotes the symmetrization of $\nabla_g$
  acting on symmetric 2-tensors.
\end{theorem}

This formula suggest naturally the introduction of the vector space
\begin{eqnarray*}
  \mathbbm{F}_g & \assign & \left\{ v \in C^{\infty} \left( X,
  S_{_{\mathbbm{R}}}^2 T^{\ast}_X \right) \mid \hspace{0.25em} \nabla_{_{T_X,
  g}} \,v_g^{\ast} \;\,=\;\, 0\, \right\}\;,
\end{eqnarray*}
where $\nabla_{_{T_X, g}}$ denotes the covariant exterior derivative acting on
$T_X$-valued differential forms and $v^{\ast}_g \assign g^{- 1} v$ denotes the endomorphism section
  associated to $v$. Indeed we observe the following corollary.

\begin{corollary}
  \label{cor-II-Var-W}Let $(X, g)$ be a compact and orientable Riemannan
  manifold and let $\Omega > 0$ be a smooth volume form. Then for all $v \in
  \mathbbm{F}_g$ hold the second variation formula
  \begin{equation}
    \label{part-vr-W} \nabla_G\, D\, \mathcal{W}_{\Omega} \,(g)\, (v, v) = \int_X \left[
    \left\langle v \cdot \tmop{Ric}^{\ast}_g (\Omega)\,,\, v \right\rangle_g \;\,+\;\,
    \frac{1}{2}\; \big| \nabla_g \,v\big|^2_g \right] \Omega \;.
  \end{equation}
  In particular if $g$ satisfies the inequality $\tmop{Ric}_g (\Omega)
  \geqslant \varepsilon g,$ for some $\varepsilon \in \mathbbm{R}_{> 0},$ then
  the bilinear form
  \begin{eqnarray*}
    \nabla_G \,D\, \mathcal{W}_{\Omega} \,(g)  \;: \; \mathbbm{F}_g \times
    \mathbbm{F}_g \;\longrightarrow \;\mathbbm{R}\;,
  \end{eqnarray*}
  is positive definite and for all $v \in \mathbbm{F}_g$ hold the inequality
  \begin{eqnarray}
    \label{ineq-vr-W}\nabla_G \,D\, \mathcal{W}_{\Omega} \,(g)\, (v, v) \;\; \geqslant \;\; \int_X \left[
    \varepsilon \,|v|^2_g \;\,+\;\, \frac{1}{2}\; \big| \nabla_g\, v\big|^2_g  \right] \Omega
    \;\;\geqslant\;\; 0 \;.
  \end{eqnarray}
\end{corollary}

An other surprising consequence of the proof of Theorem \ref{lm-IIvr-W}
is a drastically simple second variation formula for the functional
$\mathcal{W}_{\Omega}$ over the space of K\"ahler metrics with respect to a
fixed complex structure.

\begin{lemma}
  \label{Kah-vr-W-fxcx}Let $(X, g, J)$ be a compact K\"ahler manifold and let
  $\Omega > 0$ be a smooth volume form. For all $J$-invariant $v \in C^{\infty}
  \left( X, S_{_{\mathbbm{R}}}^2 T^{\ast}_X \right)$ such that the
  differential form $v J$ is $d$-closed hold the second variation formula
  \begin{eqnarray*}
    \nabla_G D \mathcal{W}_{\Omega} (g) (v, v) & = & \int_X \left\langle v
    \cdot \tmop{Ric}^{\ast}_g (\Omega)\,,\, v \right\rangle_g \Omega \;.
  \end{eqnarray*}
\end{lemma}

In order to obtain our general second variation formula for the functional
$\mathcal{W}_{\Omega}$ with respect to more general variations of K\"ahler
structures we need to perform a detailed study of their properties. This is
done in detail in the section \ref{Prop-vr-KStr}. Our main result is the
following.

\begin{theorem}
  {\bf{(Main Result)}}. \label{lm-Kah-IIvr-W}Let $(X, g,
  J)$ be a compact K\"ahler manifold, let $\Omega > 0$ be a smooth volume form
  and let $(g_t, J_t)_{t \in (- \varepsilon,
  \varepsilon)} \text{}, g = g_0, J = J_0$ be a smooth family of K\"ahler
  structures such that $\dot{J}_t = ( \dot{J}_t)_{g_t}^T .$ Then at the point
  $g$ hold the second variation formula in the direction $v \assign \dot{g}_0$
  \begin{eqnarray*}
    \nabla_G \,D \,\mathcal{W}_{\Omega} (g) (v, v) & = & \int_X \left\{
    \tmop{Tr}_{_{\mathbbm{R}}} \left[ (v_g^{\ast})^2 \tmop{Ric}^{\ast}_g
    (\Omega) \right] \;+\; \frac{1}{2}\; \big| \nabla^{0, 1}_{g, J}
    \,(v_g^{\ast})_{_J}^{0, 1} \big|_g^2\right\} \Omega 
\\
\\
    & - & \frac{1}{2} \int_X \left\langle 4 \,\overline{\partial}_{_{T_{X, J}}}
    (v_g^{\ast})_{_J}^{1, 0} \,+\,\partial^g_{_{T_{X, J}}} (v_g^{\ast})_{_J}^{0,
    1},\, \partial^g_{_{T_{X, J}}} (v_g^{\ast})_{_J}^{0, 1} \right\rangle_g
    \Omega\,,
  \end{eqnarray*}
  where $(v_g^{\ast})_{_J}^{1, 0}$, $(v_g^{\ast})_{_J}^{0, 1}$
  denote respectively the $J$-linear and $J$-anti-linear components of the endomorphism section
  $v^{\ast}_g$. In particular if $v \in \mathbbm{F}_g$ then hold the identity
  \begin{eqnarray*}
    \nabla_G \,D\, \mathcal{W}_{\Omega}\, (g) \,(v, v) & = & \int_X 
    \tmop{Tr}_{_{\mathbbm{R}}} \left[ (v_g^{\ast})^2 \tmop{Ric}^{\ast}_g
    (\Omega) \right] \Omega
\\
\\
& +& \int_X \left[\,\frac{1}{2}\; \big| \nabla_g (v_g^{\ast})_{_J}^{0, 1} \big|_g^2 \;\,+\;\,
    \big| \nabla^{0, 1}_{g, J} (v_g^{\ast})_{_J}^{0, 1} \big|_g^2\right] \Omega \;. 
  \end{eqnarray*}
\end{theorem}

We observe that the variations of K\"ahler structures such that $\dot{J}_t = (
\dot{J}_t)_{g_t}^T$ are quite natural in K\"ahler geometry (see \cite{Don} for example). Moreover from the point
of view of applications is seems to be useless to consider more general type
of variations. It is easy to see that for this type of variations hold the
identity $J \dot{J}_0 = (v_g^{\ast})_{_J}^{0, 1}$. (See the section
\ref{Prop-vr-KStr}.) This shows that the main therm in the second variation expression depends 
strongly on the variation of the complex structure.
We point out also that the $\mathbbm{F}_g$-valued variations of K\"ahler
structures enjoy the following remarkable property.

\begin{proposition}
  Let $(g_t)_{t \geqslant 0}$ be a smooth family of Riemannian metrics such
  that $\dot{g}_t \in \mathbbm{F}_{g_t}$ for all $t \geqslant 0$ and let
  $(J_t)_{t \geqslant 0}$ be a family of endomorphism sections of $T_X$ solution of
  the $\tmop{ODE}$
  \[ 2 \;\dot{J}_t \;\;=\;\; J_t  \,\dot{g}_t^{\ast} \;\, - \;\,
     \dot{g}_t^{\ast} J_t\;, \]
  with K\"ahler initial data $(J_0, g_0)$. Then $(J_t, g_t)_{t \geqslant 0}$
  is a smooth family of K\"ahler structures.
\end{proposition}
Relevant applications of this fact will be presented in a forthcoming work. We
expect that our main result will imply quite sharp stability statements. We
postpone this considerations in a forthcoming study.

\section{The first variation of the $\Omega$-Bakry-Emery-Ricci tensor}

Let $\Omega > 0$ be a smooth volume form over an oriented Riemannian manifold
$(X, g)$. We remind that the $\Omega$-Bakry-Emery-Ricci tensor of $g$ is
defined by the formula
\[ \tmop{Ric}_g (\Omega) \hspace{0.75em} : = \hspace{0.75em} \tmop{Ric} (g)
   \hspace{0.75em} + \hspace{0.75em} \nabla_g \,d\, \log \frac{dV_g}{\Omega} \;. \]
A Riemannian metric $g$ is called a $\Omega$-Shrinking Ricci soliton if $g = \tmop{Ric}_g (\Omega)$. We define the the
symmetrization of $\nabla_g$ acting on
symmetric tensors as the operator
\[ \hat{\nabla}_g : C^{\infty} \left( S^p T^{\ast}_M \right) \longrightarrow
   C^{\infty} \left( S^{p + 1} T^{\ast}_M \right)  \hspace{0.25em}, \]
\begin{eqnarray*}
  \hat{\nabla}_g  \hspace{0.25em} \alpha \,(\xi_0, ..., \xi_p) & = & \sum_{j =
  0}^p \nabla_g \,\alpha \,(\xi_j, \xi_0, ..., \hat{\xi}_j, ..., \xi_p) \;.
\end{eqnarray*}
We observe in fact that
\[ \nabla_g : C^{\infty} \left( S^p T^{\ast}_M \right) \longrightarrow
   C^{\infty} \left( T^{\ast}_M \otimes_{_{\mathbbm{R}}} S^p T^{\ast}_M
   \right) \hspace{0.25em} . \]
With this notations hold the following variation result.
\begin{lemma}
  \label{var-Om-Ric}Let $(g_t )_t$ be a smooth family of Riemannian
  metrics over an orientable manifold oriented by a smooth volume form $\Omega > 0$. Then
  hold the first variation formula
  \[ 2 \,\frac{d}{d t} \tmop{Ric}_{g_t} (\Omega) \;\;=\;\; e^{f_t} \tmop{div}_{g_t}
     \left( e^{- f_t} \mathcal{D}_{g_t}  \dot{g}_t \right), \]
  where $f_t \;\assign\; \log \frac{d V_{g_t}}{\Omega}$ and $\mathcal{D}_g \;\assign\;
  \hat{\nabla}_g \;-\; 2 \,\nabla_g$.
\end{lemma}

\begin{proof}
  A very large part of the proof is taken from a standard computation in \cite{Bes}. We include it here for
readers convenience.
We remind (see \cite{Bes}) that the first variation $\dot{\nabla}_{g_t}$ of the connection is given by the formula
  \begin{equation}
    \label{Bess-var-LC-conn} 2 \,g_t \left( \dot{\nabla}_{g_t} (\xi, \eta), \mu
    \right) = \nabla_{g_t} \dot{g}_t (\xi, \eta, \mu) \hspace{0.75em} +
    \hspace{0.75em} \nabla_{g_t} \dot{g}_t (\eta, \xi, \mu) \hspace{0.75em} -
    \hspace{0.75em} \nabla_{g_t} \dot{g}_t (\mu, \xi, \eta) \hspace{0.25em} .
  \end{equation}
  This rewrites as
  \begin{eqnarray*}
    2 \,g_t \left( \dot{\nabla}_{g_t} (\xi, \eta), \mu \right) & = &
    \hat{\nabla}_{g_t} \dot{g}_t (\xi, \eta, \mu) \hspace{0.75em} -
    \hspace{0.75em} 2 \,\nabla_{g_t} \dot{g}_t (\mu, \xi, \eta)\\
    &  & \\
    & = & \hat{\nabla}_{g_t} \dot{g}_t (\mu, \xi, \eta) \hspace{0.75em} -
    \hspace{0.75em} 2 \,\nabla_{g_t} \dot{g}_t (\mu, \xi, \eta) \hspace{0.25em},
  \end{eqnarray*}
  thanks to the symmetry properties of $\hat{\nabla}_{g_t}\, \dot{g}_t$. We infer the identity
  \begin{equation}
    \label{var-LC-conn} 2 \,g_t \left( \dot{\nabla}_{g_t} (\xi, \eta), \mu
    \right) \hspace{0.75em} = \hspace{0.75em} \mathcal{D}_{g_t} \dot{g}_t
    \hspace{0.25em} (\mu, \xi, \eta) \;.
  \end{equation}
  Taking a covariant derivative of this identity we obtain
  \begin{equation}
    \label{Der-var-LC-conn} 2\, g_t ( \nabla_{g_t} \dot{\nabla}_{g_t} (\zeta,
    \xi, \eta), \mu) = \nabla_{g_t} \mathcal{D}_{g_t}  \dot{g}_t (\zeta, \mu,
    \xi, \eta) .
  \end{equation}
  We observe now
  \[ \left( \frac{d}{dt} \hspace{0.25em} \tmop{Ric} (g_t) \right) (\xi, \eta)
     \hspace{0.75em} = \hspace{0.75em} \tmop{Tr}_{_{\mathbbm{R}}} \left[
     \dot{\mathcal{R}}_{g_t} (\cdot, \xi) \eta \right] \hspace{0.75em} =
     \hspace{0.75em} \sum_{j = 1}^n g \left( \dot{\mathcal{R}}_{g_t} (e_j,
     \xi) \eta, e_j \right) \hspace{0.25em}, \]
  for any $g_t$-orthonormal frame $(e_j)_j$. This combined with the variation
  identity
  \begin{equation}
    \label{var-Rm}  \dot{\mathcal{R}}_{g_t} (\xi, \eta) \mu \hspace{0.75em} =
    \hspace{0.75em} \nabla_{g_t} \dot{\nabla}_{g_t} (\xi, \eta, \mu)
    \hspace{0.75em} - \hspace{0.75em} \nabla_{g_t} \dot{\nabla}_{g_t} (\eta,
    \xi, \mu) \hspace{0.25em},
  \end{equation}
  (see \cite{Bes}) implies the expression
  \begin{eqnarray*}
    2 \,g_t \left( \dot{\mathcal{R}}_{g_t} (e_j, \xi) \eta, e_j \right) & = & 2\,
    g_t \left( \nabla_{g_t} \dot{\nabla}_{g_t} (e_j, \xi, \eta), e_j \right)
    \\
\\
& -& 2\, g_t \left( \nabla_{g_t}
    \dot{\nabla}_{g_t} (\xi, e_j, \eta), e_j \right)\\
    &  & \\
    & = & \nabla_{g_t}  \mathcal{D}_{g_t} \dot{g}_t (e_j, e_j, \xi, \eta)
    \hspace{0.75em} - \hspace{0.75em} \nabla_{g_t}  \mathcal{D}_{g_t}
    \dot{g}_t (\xi, e_j, e_j, \eta) \hspace{0.25em},
  \end{eqnarray*}
  thanks to the identity (\ref{Der-var-LC-conn}). Let $(x_1, ..., x_n)$ be
  $g_t$-geodesic coordinates centered at an arbitrary point $p$ and set $e_k :
  = \frac{\partial}{\partial x_k}$. Then the local frame $(e_k)_k$ is $g_t
  (p)$-orthonormal at the point $p$ and satisfies $\nabla_{g_t} e_j (p) = 0$
  for all $j$. We take now the vector fields $\xi$ and $\eta$ with constant
  coefficients with respect to the $g_t$-geodesic coordinates $(x_1, ...,
  x_n)$. Therefore $\nabla_{g_t} \xi (p) = \nabla_{g_t} \eta (p) = 0$. We
  infer the identity at the space time point $(p,t)$
  \begin{eqnarray*}
    \nabla_{g_t}  \mathcal{D}_{g_t}  \dot{g}_t (\xi, e_j, e_j, \eta) & = &
    \nabla_{g_t, \xi} \hspace{0.25em} \left[ \hspace{0.25em} \mathcal{D}_{g_t}
    \dot{g}_t (e_j, g_t^{- 1} e^{\ast}_j, \eta) \right]\\
    &  & \\
    & = & \nabla_{g_t, \xi} \hspace{0.25em} \left[ \left( \tmop{Tr}_{g_t}
    \hspace{0.25em} \mathcal{D}_{g_t}  \dot{g}_t \right) (\eta) \right] .
  \end{eqnarray*}
  We observe indeed the trivial identities
  \begin{eqnarray*}
    \tmop{Tr}_{g_t}  \dot{g}_t & = & \tmop{Tr}_{_{\mathbbm{R}}} \left( g_t^{-
    1} \dot{g}_t \right)\\
    &  & \\
    & = & \sum_{j = 1}^n e^{\ast}_j \left( g_t^{- 1} \dot{g}_t \right) e_j\\
    &  & \\
    & = & \sum_{j = 1}^n g_t \left( g_t^{- 1} e^{\ast}_j, \left( g_t^{- 1}
    \dot{g}_t \right) e_j \right)\\
    &  & \\
    & = & \sum_{j = 1}^n \dot{g}_t (e_j, g_t^{- 1} e_j^{\ast})
    \hspace{0.25em} .
  \end{eqnarray*}
  Deriving this last identity we get the formula $\eta . \tmop{Tr}_{g_t} 
  \dot{g}_t = \tmop{Tr}_g \nabla_{g_t, \eta}  \dot{g}_t$. Moreover the
  identity (\ref{var-LC-conn}) combined with (\ref{Bess-var-LC-conn}) gives
  \begin{eqnarray*}
    \mathcal{D}_{g_t}  \dot{g}_t (e_j, e_j, \eta) & = & \nabla_{g_t} \dot{g}_t
    (e_j, e_j, \eta) \hspace{0.75em} + \hspace{0.75em} \nabla_{g_t} \dot{g}_t
    (\eta, e_j, e_j) \hspace{0.75em} - \hspace{0.75em} \nabla_{g_t} \dot{g}_t
    (e_j, e_j, \eta)\\
    &  & \\
    & = & \nabla_{g_t} \dot{g}_t (\eta, e_j, e_j) .
  \end{eqnarray*}
  We deduce the identity $\left( \tmop{Tr}_{g_t} 
  \mathcal{D}_{g_t} \dot{g}_t \right) (\eta) \;=\; \tmop{Tr}_g \nabla_{g_t, \eta}
  \,\dot{g}_t$ and thus 
$$
\tmop{Tr}_{g_t}  \mathcal{D}_{g_t} 
  \dot{g}_t \;\;=\;\; d \tmop{Tr}_{g_t}  \dot{g}_t\;.
$$
We obtain the variation formula
  \begin{equation}
    \label{var-Ric} 2 \,\frac{d}{d t} \,\tmop{Ric} (g_t) \;\;=\;\; \tmop{div}_{g_t}
    \mathcal{D}_{g_t}^{}  \dot{g}_t \;\,-\;\, \nabla_{g_t} d \tmop{Tr}_{g_t} 
    \dot{g}_t\; .
  \end{equation}
  We observe now the identity
  \[ 2 \,  \frac{d}{dt} \, \nabla_{g_t} d\,f_t
     \hspace{0.75em} = \hspace{0.75em} 2 \hspace{0.25em} \dot{\nabla}_{g_t} d
     \,f_t \;\, + \;\, 2 \, \nabla_{g_t} d
     \,\dot{f}_t \;.  \]
  Time deriving the definition of the dual connection we obtain the identity
  \[ ( \dot{\nabla}_{g_t, \,\xi} \alpha) \cdot \eta \;\;=\;\; -\;\, \alpha \cdot
     \dot{\nabla}_{g_t, \,\xi} \hspace{0.25em} \eta\;, \]
  for any $1$-form $\alpha$. By applying this identity to $\alpha = d\,f_t$ we
  infer
  \begin{eqnarray*}
    2 \, \dot{\nabla}_{g_t} d\,f_t (\xi, \eta) & = & -
   \;\, 2 \, d\,f_t \cdot \dot{\nabla}_{g_t} (\xi,
    \eta) 
\\
\\
& =&  - \;\, 2\, g_t \left(
    \dot{\nabla}_{g_t} (\xi, \eta), \nabla_{g_t} f_t \right)\\
    &  & \\
    & = & - \hspace{0.25em} \mathcal{D}_{g_t}  \dot{g}_t (\nabla_{g_t} f_t,
    \xi, \eta) \hspace{0.25em},
  \end{eqnarray*}
  i.e. $2\, \dot{\nabla}_{g_t} d\,f_t \;=\; -\, \nabla_{g_t} f_t \,\neg\,
  \mathcal{D}_{g_t}  \dot{g}_t$, where $\neg$ denotes the contraction operation. Thus hold the formula
  \[ 2 \, \frac{d}{dt} \, \nabla_{g_t} d\,f_t
     \;\; = \;\; - \;\,\nabla_{g_t} f_t \;\neg\;
     \mathcal{D}_{g_t} \dot{g}_t \;\, + \;\,
     \nabla_{g_t} d \tmop{Tr}_{g_t} \dot{g}_t \;, \]
  since $2 \,\dot{f}_t \;=\; \tmop{Tr}_{g_t}  \dot{g}_t$. Using the identity
  (\ref{var-Ric}) we obtain
  \begin{eqnarray*}
    2 \, \frac{d}{dt} \,\tmop{Ric}_{g_t} (\Omega) &
    = & \tmop{div}_{g_t}  \mathcal{D}_{g_t}  \dot{g}_t \hspace{0.75em} -
    \hspace{0.75em} \nabla_{g_t} f_t \;\neg\; \mathcal{D}_{g_t}  \dot{g}_t\\
    &  & \\
    & = & \tmop{Tr}_{g_t} \left( \nabla_{g_t}  \mathcal{D}_{g_t}  \dot{g}_t
    \hspace{0.75em} - \hspace{0.75em} d\,f_t \otimes \mathcal{D}_{g_t} 
    \dot{g}_t \right)\\
    &  & \\
    & = & e^{f_t} \tmop{Tr}_{g_t} \left( e^{- f_t} \nabla_{g_t} 
    \mathcal{D}_{g_t}  \dot{g}_t \hspace{0.75em} + \hspace{0.75em} d\,e^{- f_t}
    \otimes \mathcal{D}_{g_t}  \dot{g}_t \right)\\
    &  & \\
    & = & e^{f_t} \tmop{Tr}_{g_t} \nabla_{g_t} \left( e^{- f_t} 
    \mathcal{D}_{g_t}  \dot{g}_t \right) \;.
  \end{eqnarray*}
  This concludes the proof of the lemma. 
\end{proof}

We denote by $P_g^{\ast}$ the formal adjoint of an operator $P$ with respect
to a metric $g$. We observe that the operator
\begin{eqnarray*}
  P^{\ast_{_{\Omega}}}_g & : = & e^f P^{\ast}_g  \left( e^{- f} \bullet
  \right),
\end{eqnarray*}
is the formal adjoint of $P$ with respect to the scalar product $\int_X
\left\langle \cdot, \cdot \right\rangle_g \Omega .$ With this notations hold
the following corollary

\begin{corollary}
  \label{vr-ORc-adj}Let $X$ be an orientable manifold oriented by a smooth
  volume form $\Omega > 0$ and let $(g_t )_t$ be a smooth family of
  Riemannian metrics. Then hold the first variation formula
  \[  \frac{d}{d t} \,\tmop{Ric}_{g_t} (\Omega) \;\;=\;\;
     \nabla^{\ast_{_{\Omega}}}_{g_t} \nabla_{g_t}  \dot{g}_t \;\,-\;\, \frac{1}{6} \,
     \hat{\nabla}^{\ast_{_{\Omega}}}_{g_t}  \hat{\nabla}_{g_t}  \dot{g}_t \;. \]
\end{corollary}

\begin{proof}
  We remind that $\nabla_g^{\ast} = - \tmop{div}_g$ and that
  $\hat{\nabla}^{\ast}_g = 3 \nabla_g^{\ast}$ in restriction to symmetric
  3-forms. (See the subsection \ref{apdx1} in the appendix). This last
  identity implies $\hat{\nabla}^{\ast_{_{\Omega}}}_{g_t} = 3
  \nabla^{\ast_{_{\Omega}}}_{g_t}$ in restriction to symmetric 3-forms.
\end{proof}

\section{Geodesics in the space of Riemannian metrics}

The content of this section is well known. The presentation is adapted to our particular situation. Moreover
a formula below will be needed in the next section.

{\tmstrong{The differential point of view}}. Let $\mathcal{H} : = L^2 (X,
S^2_{_{\mathbbm{R}}} T_X^{\ast})$. We equip the set $\mathcal{M} \subset
\mathcal{H}$ of smooth Riemannian metrics over $X$ with the Riemannian metric
\begin{equation}
  \label{Glb-Rm-m} G_g (u, v) = \int_X \left\langle \hspace{0.25em} u, v
  \right\rangle_g \Omega\;,
\end{equation}
for all $g \in \mathcal{\mathcal{M}}$ and $u, v \in \mathcal{H}$. We observe
now that any element $u \in T_X^{\ast} \otimes T_X^{\ast}$ can be considered
as a morphism $u : T_X \rightarrow T_X^{\ast}$. Thus we can define the
endomorphism $u^{\ast}_g : = g^{- 1} u$. With this notation hold the identity
\begin{equation}
  \label{prod-Tr}  \left\langle u, v \right\rangle_g \;\;=\;\;
  \tmop{Tr}_{_{\mathbbm{R}}} (u_g^{\ast} \hspace{0.25em} v^{\ast}_g)\;,
\end{equation}
for all $u, v \in S^2_{_{\mathbbm{R}}} T_X^{\ast}$. Using the identity
(\ref{prod-Tr}) we compute now the Fr\'echet derivative
\[ D_g G : \mathcal{H} \times \mathcal{H} \rightarrow \mathcal{H}^{\ast}
   \hspace{0.25em}, \]
of the metric $G$ at a point $g \in \mathcal{M}$. For this purpose let
$(g_s)_{s \in (- \varepsilon, \varepsilon)} \subset \mathcal{M}$ be a smooth
curve and for notation simplicity let denote $u^{\ast}_s : = g^{- 1}_s u$.
Then hold the equalities
\begin{eqnarray*} \frac{d}{ds}  \left\langle \hspace{0.25em} u, v \right\rangle_{g_s} 
&=&
   \frac{d}{ds} \hspace{0.25em} \tmop{Tr}_{_{\mathbbm{R}}}
   (u_s^{\ast} \hspace{0.25em} v^{\ast}_s) 
\\
\\
& =& 
   \tmop{Tr}_{_{\mathbbm{R}}} \left( \frac{d}{ds}  \hspace{0.25em} u_s^{\ast}
   \hspace{0.25em} v^{\ast}_s \hspace{0.75em} + \hspace{0.75em} u_s^{\ast}
   \hspace{0.25em}  \frac{d}{ds} \hspace{0.25em} v^{\ast}_s \right)
   \;, 
\end{eqnarray*}
and $\frac{d}{ds} \hspace{0.25em} u_s^{\ast} = - \dot{g}_s^{\ast}
\hspace{0.25em} u_s^{\ast}$ since $\frac{d}{ds} \hspace{0.25em} g_s^{- 1} = -
g_s^{- 1} \hspace{0.25em} \dot{g}_s \hspace{0.25em} g_s^{- 1}$. 
Thus
\begin{eqnarray*}
\frac{d}{ds}  \left\langle \hspace{0.25em} u, v \right\rangle_{g_s}
   & = & - \;\,
   \tmop{Tr}_{_{\mathbbm{R}}} ( \dot{g}_s^{\ast} \hspace{0.25em} u_s^{\ast}
   \hspace{0.25em} v^{\ast}_s \hspace{0.75em} + \hspace{0.75em} u_s^{\ast}
   \hspace{0.25em}  \dot{g}_s^{\ast} \hspace{0.25em} v_s^{\ast})
   \\
\\
&=&  - \;\, 2\,
   \tmop{Tr}_{_{\mathbbm{R}}} ( \dot{g}_s^{\ast} \hspace{0.25em} u_s^{\ast}
   \hspace{0.25em} v^{\ast}_s) \hspace{0.25em}, 
\end{eqnarray*}
since $\dot{g}_s$ is also symmetric. Indeed we observe the elementary
identities
\begin{eqnarray*}
  \tmop{Tr}_{_{\mathbbm{R}}} \left[ (u_s^{\ast} \hspace{0.25em} 
  \dot{g}_s^{\ast}) v_s^{\ast} \right] & = & \tmop{Tr}_{_{\mathbbm{R}}} \left[
  v_s^{\ast} (u_s^{\ast} \hspace{0.25em}  \dot{g}_s^{\ast}) \right]\\
  &  & \\
  & = & \tmop{Tr}_{_{\mathbbm{R}}} \left[ v_s^{\ast} (u_s^{\ast}
  \hspace{0.25em}  \dot{g}_s^{\ast}) \right]_s^T\\
  &  & \\
  & = & \tmop{Tr}_{_{\mathbbm{R}}} \left[ (u_s^{\ast} \hspace{0.25em} 
  \dot{g}_s^{\ast})_s^T v^{\ast}_s \right]\\
  &  & \\
  & = & \tmop{Tr}_{_{\mathbbm{R}}} ( \dot{g}_s^{\ast} \hspace{0.25em}
  u_s^{\ast} \hspace{0.25em} v^{\ast}_s),
\end{eqnarray*}
where $A_s^T$ denotes the transpose of $A$ with respect to $g_s$. So if we set
$g : = g_0$ and $h : = \dot{g}_0$ we infer the identity
\begin{eqnarray*}
D_g G (h, u) v 
&=&
- \hspace{0.75em} 2
   \hspace{0.25em} \int_X \tmop{Tr}_{_{\mathbbm{R}}} (h_g^{\ast}
   \hspace{0.25em} u_g^{\ast} \hspace{0.25em} v^{\ast}_g) \hspace{0.25em}
   \Omega 
\\
\\
&=& 
- \hspace{0.75em} 2 \hspace{0.25em} \int_X \left\langle h
   u_g^{\ast} \hspace{0.25em}, v \right\rangle_g \Omega \hspace{0.25em} .
\end{eqnarray*}
Clearly the domain of $D_g G$ is of type $E \times \mathcal{H}$, with $E
\subset \mathcal{H}$ a linear space dense inside $\mathcal{H}$. We remind now
that the Levi-Civita connection $\nabla_G = D + \Gamma_G$ of $G$ is given by
the formula
\[  2 \hspace{0.25em}  \Gamma_G (g) (u, v) \hspace{0.75em} : =
   \hspace{0.75em} G_g^{- 1} \Big[ D_g G (u, v) \hspace{0.75em} + \hspace{0.75em}
   D_g G (v, u) \hspace{0.75em} - \hspace{0.75em} D_g G (\bullet, u) v \Big]
   \hspace{0.25em} . \]
We explicit the therm
\begin{eqnarray*}
  &  & D_g G (u, v) \hspace{0.75em} + \hspace{0.75em} D_g G (v, u)
  \hspace{0.75em} - \hspace{0.75em} D_g G (\bullet, u) v\\
  &  & \\
  & = & - \hspace{0.75em} 2 \hspace{0.25em} \int_X \tmop{Tr}_{_{\mathbbm{R}}}
  \left( u_g^{\ast} \hspace{0.25em} v^{\ast}_g \hspace{0.25em}
  \bullet^{\ast}_g \hspace{0.75em} + \hspace{0.75em} v_g^{\ast}
  \hspace{0.25em} u^{\ast}_g \hspace{0.25em} \bullet^{\ast}_g \hspace{0.75em}
  - \hspace{0.75em} \bullet^{\ast}_g \hspace{0.25em} u_g^{\ast}
  \hspace{0.25em} v^{\ast}_g \right) \hspace{0.25em} \Omega \hspace{0.25em}\\
  &  & \\
  & = & - \hspace{0.75em} 2 \hspace{0.25em} \int_X \tmop{Tr}_{_{\mathbbm{R}}}
  (u_g^{\ast} \hspace{0.25em} v^{\ast}_g \hspace{0.25em} \bullet^{\ast}_g)
  \hspace{0.25em} \Omega \hspace{0.25em}\\
  &  & \\
  & = & - \hspace{0.25em} \int_X \tmop{Tr}_{_{\mathbbm{R}}} \left( u_g^{\ast}
  \hspace{0.25em} v^{\ast}_g \hspace{0.25em} \bullet^{\ast}_g \hspace{0.75em}
  + \hspace{0.75em} v_g^{\ast} \hspace{0.25em} u^{\ast}_g \hspace{0.25em}
  \bullet^{\ast}_g \right) \hspace{0.25em} \Omega \hspace{0.25em}\\
  &  & \\
  & = & - \hspace{0.25em} \int_X \left\langle u v^{\ast}_g \hspace{0.25em} \;\,+\;\,
  v u^{\ast}_g\,, \bullet \right\rangle_g \Omega\; .
\end{eqnarray*}
We infer the identity
\begin{equation}
  \label{Krist-G}
2 \,\Gamma_G (g) (u, v) \;\;=\;\; -\;\, u \,v^{\ast}_g \;\,-\;\, v \,u^{\ast}_g \;.
\end{equation}
We deduce that the equation of the geodesics $\ddot{g}_t + \Gamma_G (g_t) (
\dot{g}_t, \dot{g}_t) = 0 \hspace{0.25em},$ writes as
$$
\ddot{g}_t \hspace{0.75em} - \hspace{0.75em} \dot{g}_t
   \hspace{0.25em} g_t^{- 1} \dot{g}_t \hspace{0.75em} = \hspace{0.75em} 0
   \hspace{0.25em},
$$
(we consider $g_t$ as a morphism $g_t : T_X \rightarrow T_X^{\ast}$) or as
$\ddot{g}_t^{\ast} - ( \dot{g}_t^{\ast})^2 = 0 .$ Moreover time deriving the
identity $\dot{g}_t = g_t \hspace{0.25em} \dot{g}_t^{\ast}$, which defines the
endomorphism $ \dot{g}_t^{\ast},$we obtain the expression
\[ \ddot{g}_t \hspace{0.75em} = \hspace{0.75em} \dot{g}_t \hspace{0.25em}
   \dot{g}_t^{\ast} \hspace{0.75em} + \hspace{0.75em} g_t \hspace{0.25em} 
   \frac{d}{dt} \hspace{0.25em}  \dot{g}_t^{\ast} \hspace{0.25em}, \]
and thus
\begin{equation}
  \label{Inv-geo-endm} 0 \;\;=\;\; \ddot{g}_t^{\ast} \hspace{0.75em} - \hspace{0.75em}
  ( \dot{g}_t^{\ast})^2 \hspace{0.75em} = \hspace{0.75em} \frac{d}{dt}
  \hspace{0.25em} \dot{g}_t^{\ast} .
\end{equation}
But the equation (\ref{Inv-geo-endm}) implies the identity $g_t^{- 1}
\dot{g}_t = g_0^{- 1} \dot{g}_0$. Thus the geodesic writes explicitly as
\begin{equation}
  \label{expr-geod} g_t \hspace{0.75em} = \hspace{0.75em} g_0 \hspace{0.25em}
  e^{\hspace{0.25em} tg_0^{- 1} \dot{g}_0} \hspace{0.25em} .
\end{equation}
We show in the subsection \ref{apdx3} of the appendix that the Riemannian space $(\mathcal{M}, G)$ is non-positively
curved. 

{\tmstrong{The metric point of view}}. For any $g_0, g_1 \in \mathcal{M}$ we
consider the set of curves
\begin{eqnarray*}
  C^1 (g_0, g_1) & \assign & \left\{ h \in C^1 ([0, 1], \mathcal{M}) \mid\; h_0
  \;\;=\;\; g_0\;,\;\; h_1\;\; =\;\; g_1 \right\}\;,
\end{eqnarray*}
and the distance function
\begin{eqnarray*}
  d_G (g_0, g_1) & \assign & \inf_{h \in C^1 (g_0, g_1)}  \int^1_0 d t \left[
  \int_X | \dot{h}_t |^2_{h_t} \Omega \right]^{1 / 2} .
\end{eqnarray*}
Let $(g_t)_{t \in \mathbbm{R}} \subset \mathcal{M}$ be the unique geodesic
joining $g_0$ to $g_1$. Then
\begin{eqnarray*}
  d_G (g_0, g_1) & = &  \left[ \int_X | \dot{g}_0 |^2_{g_0} \Omega \right]^{1
  / 2} \\
  &  & \\
  & = & \left[ \int_X \tmop{Tr}_{_{\mathbbm{R}}} (g^{- 1}_0  \dot{g}_0)^2
  \Omega \right]^{1 / 2}\\
  &  & \\
  & = & \left[ \int_X \tmop{Tr}_{_{\mathbbm{R}}} \left[ \log (g^{- 1}_0 g_1)
  \right]^2 \Omega \right]^{1 / 2},
\end{eqnarray*}
thanks to the expression of the geodesics (\ref{expr-geod}). We observe now that a
sequence $\{g_k\}_k \subset \mathcal{M}$ is $d_G$-Cauchy if and only if the sequence $\{\log
(g^{- 1}_0 g_k)\}_k$ is $L^2$-Cauchy. Indeed hold the identity
\begin{eqnarray*}
  d_G (g_k, g_{k + l}) & = & \left[ \int_X \tmop{Tr}_{_{\mathbbm{R}}} \left[
  \log (g^{- 1}_0 g_{k + l}) \;\,-\;\, \log (g^{- 1}_0 g_k) \right]^2 \Omega
  \right]^{1 / 2} .
\end{eqnarray*}
This is because for any $g_0$-symmetric endomorphisms $A, B, C$, with $A, B >
0$, the identity 
$$
\tmop{Tr}_{_{^{\mathbbm{R}}}} (A B C) \;\;=\;\;\tmop{Tr}_{_{^{\mathbbm{R}}}} (B A C)\;,
$$
implies
\begin{eqnarray*}
  \tmop{Tr}_{_{^{\mathbbm{R}}}} \left[ \log (A B) C \right] & = &
  \tmop{Tr}_{_{^{\mathbbm{R}}}} \left[ (\log A \;\,+\;\, \log B) C \right]\; .
\end{eqnarray*}
We obtain
\begin{eqnarray*}
  d_G (g_k, g_{k + l}) & = & \left[ \int_X \big| \log (g^{- 1}_0 g_{k + l}) \;\,-\;\, \log
  (g^{- 1}_0 g_k) \big|^2_{g_0} \Omega \right]^{1 / 2}\;,
\end{eqnarray*}
since the endomorphism 
$$
\log (g^{- 1}_0 g_{k + l}) \;\,-\;\, \log (g^{- 1}_0 g_k)\;,
$$ 
is
$g_0$-symmetric. We infer that the metric completion $\mathcal{M}^{d_G}$ of
$(\mathcal{M}, d_G)$ is given by
\begin{eqnarray*}
  \mathcal{M}^{d_G} \;\; \equiv \;\; \Big\{ g \in \tmop{Mes} (X,
  S_{_{\mathbbm{R}}}^2 T^{\ast}_X)_{/ a.e} \mid \;g\;\, \geqslant\;\, 0\;,\; \log (g^{- 1}_0
  g) \in L^2 (X, \tmop{End} (T_X)) \Big\} \;,
\end{eqnarray*}
where the notation $/a.e$ means the almost everywhere equivalence relation.
Moreover is clear that 
$$
\mathcal{M}^{d_G}\;\;=\;\;\overline{\mathcal{M}} ^{d_G}\;.
$$
Using quite elementary relaxation considerations we can show that the metric space 
$(\mathcal{M}^{d_G},d_G)$ is a non positively curved length space in the sense of Alexandrov.
\section{The total second variation of Perelman's $\mathcal{W}$ functional}

We give now a proof theorem \ref{lm-IIvr-W}.

\begin{proof}
  Let $h_t \assign g_t - \tmop{Ric}_{g_t} (\Omega) .$ The computation of the
  geodesic equation done in the previous section shows
  \[ \nabla_G  \,\dot{g}_t \,( \dot{g}_t) \hspace{0.75em} = \hspace{0.75em}
     \ddot{g}_t \hspace{0.75em} + \hspace{0.75em} \Gamma_G (g_t) ( \dot{g}_t,
     \dot{g}_t) \hspace{0.75em} = \hspace{0.75em} \ddot{g}_t \hspace{0.75em} -
     \hspace{0.75em} \dot{g}_t \hspace{0.25em} \dot{g}_t^{\ast}
     \hspace{0.25em} . \]
  We infer that the Hessian of $\mathcal{W}_{\Omega}$ with respect to the
  Riemannian metric $G$ is given by
  \begin{eqnarray*}
    \nabla_G \,D\, \mathcal{W}_{\Omega} \,(g_t)\, ( \dot{g}_t, \dot{g}_t) & = &
    \frac{d^2}{dt^2} \hspace{0.25em}  \mathcal{W}_{\Omega} (g_t)
    \hspace{0.75em} - \hspace{0.75em} D \,\mathcal{W}_{\Omega} (\nabla_G \,
    \dot{g}_t \,( \dot{g}_t))\\
    &  & \\
    & = & \frac{d}{dt} \hspace{0.25em} \int_X \left\langle \dot{g}_t
    \hspace{0.25em}, \hspace{0.25em} h_t \right\rangle_{g_t} \Omega
    \hspace{0.75em} - \hspace{0.75em} \int_X  \left\langle \nabla_G  \,\dot{g}_t\,
    ( \dot{g}_t) \hspace{0.25em}, \hspace{0.25em} h_t \right\rangle_{g_t}
    \Omega\\
    &  & \\
    & = & \int_X \frac{d}{dt}  \left\langle \dot{g}_t \hspace{0.25em},
    \hspace{0.25em} h_t  \right\rangle_{g_t} \Omega 
\\
\\
& -&
    \int_X \tmop{Tr}_{_{\mathbbm{R}}} \Big[ \left(
    \ddot{g}_t^{\ast} \hspace{0.75em} - \hspace{0.75em} ( \dot{g}_t^{\ast})^2
    \right) h_t^{\ast} \Big] \hspace{0.25em} \Omega\\
    &  & \\
    & = & \int_X \frac{d}{dt} \hspace{0.25em} \tmop{Tr}_{_{\mathbbm{R}}}
    \left(  \dot{g}_t^{\ast} \hspace{0.25em} h_t^{\ast}  \right)
    \hspace{0.25em} \Omega \hspace{0.75em} - \hspace{0.75em} \int_X
    \tmop{Tr}_{_{\mathbbm{R}}} \left[ \hspace{0.25em} \frac{d}{dt}
    \hspace{0.25em}  \dot{g}_t^{\ast}  \hspace{0.25em} h^{\ast}_t \right]
    \hspace{0.25em} \Omega\\
    &  & \\
    & = & \int_X \tmop{Tr}_{_{\mathbbm{R}}} \left[ \hspace{0.25em}
    \dot{g}_t^{\ast} \hspace{0.25em}  \frac{d}{dt}  \hspace{0.25em} h^{\ast}_t
    \right] \hspace{0.25em} \Omega\\
    &  & \\
    & = & \int_X \tmop{Tr}_{_{\mathbbm{R}}} \left[ \hspace{0.25em} -
    \hspace{0.25em} ( \dot{g}_t^{\ast})^2 \hspace{0.25em} h^{\ast}_t
    \hspace{0.75em} + \hspace{0.75em} \dot{g}_t^{\ast} \hspace{0.25em}
    \dot{h}_t^{\ast} \right] \hspace{0.25em} \Omega\\
    &  & \\
    & = & - \hspace{0.75em} \int_X \tmop{Tr}_{_{\mathbbm{R}}}  \left[ (
    \dot{g}_t^{\ast})^2 \hspace{0.25em} h^{\ast}_t \right] \hspace{0.25em}
    \Omega \hspace{0.75em} + \hspace{0.75em} \int_X | \dot{g}_t |^2_{g_t}
    \Omega\\
    &  & \\
    & - & \int_X \left\langle \dot{g}_t\,, \frac{d}{dt} \hspace{0.25em}
    \tmop{Ric}_{g_t} (\Omega) \right\rangle_{g_t} \Omega \hspace{0.25em} .
  \end{eqnarray*}
Using corollary \ref{vr-ORc-adj} and integrating by parts we infer
  \begin{eqnarray*}
    - \hspace{0.75em} \int_X \left\langle \dot{g}_t\,, \frac{d}{dt}
    \hspace{0.25em} \tmop{Ric}_{g_t} (\Omega) \right\rangle_{g_t} \Omega & = &
    \frac{1}{6} \hspace{0.75em} \int_X \left\langle \dot{g}_t\,,
    \hat{\nabla}_{g_t}^{\ast_{_{\Omega}}}  \hat{\nabla}_{g_t}  \,\dot{g}_t
    \right\rangle_{g_t} \Omega\\
    &  & \\
    & - & \int_X \left\langle \dot{g}_t\,, \nabla_{g_t}^{\ast_{_{\Omega}}}
    \nabla_{g_t}  \,\dot{g}_t \right\rangle_{g_t} \Omega\\
    &  & \\
    & = & \int_X \left[ \hspace{0.25em} \frac{1}{6} \hspace{0.25em} \big|
    \hat{\nabla}_{g_t}  \,\dot{g}_t \big|^2_{g_t} \hspace{0.75em} - \hspace{0.75em}
    \big| \nabla_{g_t}  \,\dot{g}_t \big|^2_{g_t} \right] \Omega \hspace{0.25em},
  \end{eqnarray*}
  which implies the required second variation formula (\ref{rm-2vr-W}).
\end{proof}

We deduce easily corollary \ref{cor-II-Var-W}.

\begin{proof}
The assumption $v \in \mathbbm{F}_g$ implies that the tensor $\nabla_g\, v$ is
$3$-symmetric and thus $\hat{\nabla}_g \,v = 3 \nabla_g \,v$. Then the variation
formula (\ref{rm-2vr-W}) implies (\ref{part-vr-W}). We show now the inequality
(\ref{ineq-vr-W}). Let $(e_j)_j$ be a $g$-orthonormal basis and observe
  that
  \begin{eqnarray*}
    \tmop{Tr}_{_{\mathbbm{R}}} \left[ (v_g^{\ast})^2 \tmop{Ric}^{\ast}_g
    (\Omega) \right] & = & \tmop{Tr}_{_{\mathbbm{R}}} \left[ v_g^{\ast}
    \tmop{Ric}^{\ast}_g (\Omega) v_g^{\ast} \right]\\
    &  & \\
    & = & g \left( v_g^{\ast} \tmop{Ric}^{\ast}_g (\Omega) v_g^{\ast} e_j,
    e_j \right)\\
    &  & \\
    & = & g \left( \tmop{Ric}^{\ast}_g (\Omega) v_g^{\ast} e_j, v_g^{\ast}
    e_j \right)\\
    &  & \\
    & = & \tmop{Ric}_g (\Omega) \left( v_g^{\ast} e_j, v_g^{\ast} e_j
    \right)\\
    &  & \\
    & \geqslant & \varepsilon\, g \left( v_g^{\ast} e_j, v_g^{\ast} e_j
    \right)\\
    &  & \\
    & = & \varepsilon\, |v|^2_g\;,
  \end{eqnarray*}
  which shows the required inequality.
\end{proof}

We show now lemma \ref{Kah-vr-W-fxcx}.

\begin{proof}
  We remind first that any smooth volume form $\Omega > 0$ over a complex
  manifold $(X, J)$ of complex dimension $n$, induces a hermitian metric
  $h_{\Omega}$ over the canonical bundle $K_{_{X, J}} \assign \Lambda_{_J}^{n,
  0} T^{\ast}_{_X} $ given by the formula
  \[ h_{\Omega} (\alpha, \beta) \;\;\assign\;\; \frac{n!\, i^{n^2} \alpha \wedge
     \overline{\beta} }{\Omega} \;. \]
  By abuse of notations we will denote by $\Omega^{- 1}$ the metric
  $h_{\Omega} .$ The dual metric $h_{\Omega}^{\ast}$ on the anti canonical
  bundle $K^{- 1}_{_{X, J}} = \Lambda_{_J}^{n, 0} T_{_X}$ is given by the
  formula
  \[ h_{\Omega}^{\ast} (\xi, \eta) \;\;=\;\; (- i)^{n^2} \Omega \left( \xi_{},
     \bar{\eta} \right) / n!\; . \]
  Again by abusing notations we denote by $\Omega$ the dual metric
  $h_{\Omega}^{\ast} .$ We define the $\Omega$-Ricci form
  \[ \tmop{Ric}_{_J} \left( \Omega \right) \;\;\assign\;\; i\,\mathcal{C}_{\Omega} 
     \big( K^{- 1}_{_{X, J}} \big) \;\;=\;\; - \;\,i\,\mathcal{C}_{\Omega^{- 1}} \big(
     K_{_{X, J}} \big)\; . \]
  In particular $\tmop{Ric}_{_J} (\omega) = \tmop{Ric}_{_J} (\omega^n) .$ We
  remind also that for any $J$-invariant K\"ahler metric $g$ the associated
  symplectic form $\omega \assign g J$ satisfies the elementary identity
  \begin{eqnarray*}
    \tmop{Ric} (g) & = & - \hspace{0.25em} \tmop{Ric}_{_J} (\omega)\, J
    \hspace{0.25em} .
  \end{eqnarray*}
  Moreover for all twice differentiable function $f$ hold the identity
  \begin{eqnarray*}
    \nabla_g \,d\,f & = & - \hspace{0.25em} \big(i\, \partial_{_J}
    \overline{\partial}_{_J} f \big)\, J \hspace{0.75em} + \hspace{0.75em} g\,
    \overline{\partial}_{_{T_{X, J}}} \nabla_g \,f\; .
  \end{eqnarray*}
We infer
  the decomposition identity
  \begin{equation}
    \label{cx-dec-Ric} \tmop{Ric}_g (\Omega) \hspace{0.75em} = \hspace{0.75em}
    - \;\, \tmop{Ric}_{_J} (\Omega) \,J \hspace{0.75em} +
    \hspace{0.75em} g \,\overline{\partial}_{_{T_{X, J}}} \nabla_g \log
    \frac{dV_g}{\Omega} \;.
  \end{equation}
  Let now $(g_t)_t$ be a smooth family of $J$-invariant K\"ahler metrics such
  that $g_0 = g$ and $\dot{g}_0 = v .$ We claim that
  \begin{eqnarray*}
    \left\langle \dot{g}_t\,, \frac{d}{dt} \hspace{0.25em} \tmop{Ric}_{g_t}
    (\Omega) \right\rangle_{g_t} & \equiv & 0 \;.
  \end{eqnarray*}
  In fact time deriving the complex decomposition formula (\ref{cx-dec-Ric})
  with respect to the evolving family $(g_t)_t$ we obtain the identity
  \begin{eqnarray*}
    \frac{d}{dt} \hspace{0.25em} \tmop{Ric}_{g_t} (\Omega) & = & \dot{g}_t \,\overline{\partial}_{_{T_{X, J}}}
    \nabla_{g_t}\, f_t\;\, +\;\, g_t \,\overline{\partial}_{_{T_{X, J}}} 
    \left( \frac{d}{dt} \hspace{0.25em} \nabla_{g_t} \,f_t \right) \;.
  \end{eqnarray*}
  Thus
  \begin{eqnarray*}
    \left\langle \dot{g}_t\,, \frac{d}{dt} \hspace{0.25em} \tmop{Ric}_{g_t}
    (\Omega) \right\rangle_{g_t} & = & \tmop{Tr}_{_{\mathbbm{R}}} \left[
    \dot{g}^{\ast}_t\, g^{- 1}_t \, \frac{d}{dt} \hspace{0.25em} \tmop{Ric}_{g_t}
    (\Omega) \right]\\
    &  & \\
    & = & \tmop{Tr}_{_{\mathbbm{R}}} \left[ ( \dot{g}^{\ast}_t)^2 \,
    \overline{\partial}_{_{T_{X, J}}} \nabla_{g_t} \,f_t \right]
\\
\\
&+&
\tmop{Tr}_{_{\mathbbm{R}}} \left[ \dot{g}^{\ast}_t \,
    \overline{\partial}_{_{T_{X, J}}} \left( \frac{d}{dt} \hspace{0.25em} \nabla_{g_t}\,
    f_t \right) \right]\\
    &  & \\
    & = & 0\;,
  \end{eqnarray*}
  since $\dot{g}_t$ is $J$-invariant and the endomorphisms
$$
\overline{\partial}_{_{T_{X, J}}} \nabla_{g_t} \,f_t\;,\qquad 
    \overline{\partial}_{_{T_{X, J}}} \left( \frac{d}{dt} \hspace{0.25em} \nabla_{g_t}\,
    f_t \right) \;,
$$
are $J$-anti-linear. The conclusion follows from the last
  expression of the second variation of the $\mathcal{W}_{\Omega}$ functional
  obtained in the proof of lemma \ref{lm-IIvr-W}. 
\end{proof}

In order to compute our general total second variation formula for Perelman's
$\mathcal{W}$-functional with respect to variations of K\"ahler structures we
need to perform first a careful study of their properties. This is done
in detail in the next section. It represent the key step which allows us to
obtain our main result.

\section{Properties of the variations of K\"ahler \\Structures}\label{Prop-vr-KStr}

Let $\mathcal{M} \subset C^{\infty} (X, S^2_{_{\mathbbm{R}}} T_X^{\ast})$ be
the space of smooth Riemannian metrics over a compact manifold $X$, let
$\mathcal{J} \subset C^{\infty} (X, \tmop{End}_{_{\mathbbm{R}}} (T_X))$ be the
set of smooth almost complex structures and let
\[ \mathcal{KS} \hspace{0.75em} : = \hspace{0.75em} \Big\{ (J, g) \in
   \mathcal{J} \times \mathcal{M} \hspace{0.25em} \mid \hspace{0.25em}
   \hspace{0.25em} g \hspace{0.75em} = \hspace{0.75em} J^{\ast} g\,J
   \hspace{0.25em}, \hspace{0.75em}  \nabla_g \,J \hspace{0.75em}
   = \hspace{0.75em} 0 \;\Big\} \hspace{0.25em}, \]
be the space of K\"ahler structures. We remind that if $A \in
\tmop{End}_{_{\mathbbm{R}}} (T_X)$ then its transposed $A^T_g$ with respect to
$g$ is given by the formula
$$
A^T_g \;\;=\;\; g^{- 1} A^{\ast} g\;.
$$
We infer that the compatibility
condition 
$
g = J^{\ast} g\,J
$
is equivalent to the condition 
$$
J^T_g \;\;=\;\; -\;\, J\;.
$$
We observe now the following elementary lemmas.
\begin{lemma}
  \label{Tng-Kah-Str}For any smooth path $ (g_t,
  J_t)_t \subset \mathcal{K}\mathcal{S} \text{}$ hold the equivalent
  identities
  \begin{equation}
    \label{der-J-inv} 2 \,( \dot{g}_t^{\ast})_{J_t}^{0, 1}\;\; =\;\; -\;\, J_t \, \dot{J}_t \;\,-\;\,
    (J_t  \dot{J}_t)_{g_t}^T\;,
  \end{equation}
  
  \begin{equation}
    \label{der-J-inv-T} ( \dot{J}_t)^T_{g_t} \hspace{0.75em} + \hspace{0.75em}
    \dot{J}_t \;\;=\;\; J_t \, \dot{g}_t^{\ast} \hspace{0.75em} -
    \hspace{0.75em} \dot{g}_t^{\ast} J_t\;,
  \end{equation}
\end{lemma}

\begin{proof}
  Time deriving the compatibility condition $g_t = J_t^{\ast} g_t J_t$ we
  obtain
  \begin{equation}
    \label{gn-der-J-inv}  \dot{g}_t \hspace{0.75em} = \hspace{0.75em}
    \dot{J}_t^{\ast} g_t J_t \hspace{0.75em} + \hspace{0.75em} J_t^{\ast}
    \dot{g}_t J_t \hspace{0.75em} + \hspace{0.75em} J_t^{\ast} g_t \dot{J}_t,
  \end{equation}
  Multiplying both l.h.s of (\ref{gn-der-J-inv}) with $g^{- 1}$ we infer
  \begin{eqnarray*}
    \dot{g}_t^{\ast} & = & \dot{J}_t^T J_t \hspace{0.75em} + \hspace{0.75em}
    J_t^T \dot{g}_t^{\ast} J_t \hspace{0.75em} + \hspace{0.75em} J_t^T
    \dot{J}_t\\
    &  & \\
    & = & - \hspace{0.75em} (J_t \dot{J}_t)^T_{g_t} \hspace{0.75em} -
    \hspace{0.75em} J \dot{g}_t^{\ast} J_t \hspace{0.75em} - \hspace{0.75em}
    J_t \dot{J}_t \hspace{0.25em},
  \end{eqnarray*}
  and thus (\ref{der-J-inv}). We observe now that the identity
  (\ref{der-J-inv}) rewrites as (\ref{der-J-inv-T}) since the endomorphism $(
  \dot{J}_t)^T_{g_t}$ is $J_t$-anti linear.
\end{proof}

\begin{lemma}
  Let $(g_t)_{t \geqslant 0}$ be a smooth family of Riemannian metrics and let
  $(J_t)_{t \geqslant 0}$ be a family of endomorphisms of $T_X$ solution of
  the $\tmop{ODE}$
  \[ 2 \,\dot{J}_t \;\;=\;\; J_t  \,\dot{g}_t^{\ast} \hspace{0.75em} - \hspace{0.75em}
     \dot{g}_t^{\ast} J_t, \]
  with initial conditions $J^2_0 \; = \; - \;\I_{T_X}$
  and $(J_0)^T_{g_0} \; = \; - \;J_0$. Then this
  conditions are preserved in time i.e. $J^2_t \; =\; -\; \I_{T_X}$ and $(J_t)^T_{g_t} \; =\; -\; J_t$ for all $t \geqslant 0$. 
\end{lemma}

\begin{proof}
  Expanding the time derivative of $J^2_t$ we get
  \begin{eqnarray*}
    2 \,\frac{d}{d t} \,J^2_t & = & J^2_t \, \dot{g}_t^{\ast} \hspace{0.75em} -
    \hspace{0.75em} \dot{g}_t^{\ast} J^2_t \;.
  \end{eqnarray*}
  Thus the condition $J^2_t \; = \; -\; \I_{T_X} $ is
  preserved for all $ t \geqslant 0$. Dualysing the evolution equation
  of $J_t$ we infer the equation
  \[ 2 \,\dot{J}_t^{\ast} \;=\; \dot{g}_t \,g^{- 1}_t\, J^{\ast}_t \;\,-\;\, J^{\ast}_t \,
     \dot{g}_t \,g^{- 1}_t \;. \]
  Thus if we set $M_t \;\assign\; g_t \,J_t\; +\; J^{\ast}_t\, g_t$, then the time
  derivative of $M_t$ expands as
  \begin{eqnarray*}
    2 \,\dot{M_t} & = & 2 \,\dot{g}_t\, J_t \;\,+\;\, g_t \,\left( J_t  \dot{g}_t^{\ast}
    \;\, - \;\, \dot{g}_t^{\ast}\, J_t \right) 
\\
\\
&+&
2\,
    \dot{J}_t^{\ast}\, g_t \;\,+\;\, 2\, J^{\ast}_t\, \dot{g}_t\\
    &  & \\
    & = & 2 \,\dot{g}_t \,J_t \;\,+\;\, g_t\, J_t  \,\dot{g}_t^{\ast} \;\, -
\;\, \dot{g}_t \,J_t
\\
\\
&+&
\dot{g}_t\, g^{- 1}_t\, J^{\ast}_t \,g_t\;\, -\;\,
    J^{\ast}_t\,  \dot{g}_t
\;\, +\;\, 2\, J^{\ast}_t\, \dot{g}_t \\
    &  & \\
    & = &  \dot{g}_t \,J_t \;\,+ J^{\ast}_t \, \dot{g}_t \;\,+\;\, g_t \,J_t\, g^{- 1}_t 
    \,\dot{g}_t \;\,+\;\, \dot{g}_t\, g^{- 1}_t\, J^{\ast}_t\, g_t\\
    &  & \\
    & = &  \dot{g}_t \,J_t \;\,+\;\, J^{\ast}_t\,  \dot{g}_t \;\,+\;\, M_t\, g^{- 1}_t \, \dot{g}_t 
\\
\\
&-&
    J^{\ast}_t \, \dot{g}_t \;\,+\;\, \dot{g}_t\, g^{- 1}_t \,M_t \;\,-\;\, \dot{g}_t\, J_t\\
    &  & \\
    & = & M_t \,g^{- 1}_t\, \dot{g}_t \;\,+\;\, \dot{g}_t\, g^{- 1}_t\, M_t\;,
  \end{eqnarray*}
  which implies $M_t \equiv 0$, by the uniqueness of the Cauchy problem since
  $M_0 = 0,$ by assumption. We infer that the condition $(J_t)^T_{g_t}
  \; = \; -\; J_t$ is also preserved for all $t
  \geqslant 0$.
\end{proof}
We
introduce now the following quite standard notations. For any section $S \in C^{\infty} (X,
E)$ of $E : = \tmop{End}_{_{\R}} (T_{_X})$ and any $\xi, \eta \in T_X$ we
define
\begin{eqnarray*}
  \nabla_{g, J}^{1, 0} \,S \,(\xi, \eta) & : = & \frac{1}{2}\,  \Big[ \nabla_g \,S
  \,(\xi, \eta) \hspace{0.75em} - \hspace{0.75em} J\, \nabla_g \,S \,(J \xi, \eta)
  \Big] \hspace{0.25em},\\
  &  & \\
  \nabla_{g, J}^{0, 1} \,S \,(\xi, \eta) & : = &  \frac{1}{2}\,  \Big[ \nabla_g \,S
  \,(\xi, \eta) \hspace{0.75em} + \hspace{0.75em} J\, \nabla_g \,S \,(J \xi, \eta)
  \Big] \hspace{0.25em} .
\end{eqnarray*}
We define also the $J$-anti linear operator 
$$
\nabla_{g, J}^{0, 1} \,S \cdot
\hspace{0.25em} \eta \;\;: =\;\; \nabla_{g, J}^{0, 1} \,S \,(\cdot, \eta)\;.
$$ 
For any $J \in \mathcal{J}$ and any $J$-invariant $g \in \mathcal{M}$ we
define the vector space
\begin{eqnarray*}
  \mathbbm{D}^J_g & \assign & \Big\{ v \in C^{\infty} \left( X,
  S_{_{\mathbbm{R}}}^2 T^{\ast}_X \right) \mid \hspace{0.25em} \nabla_{g,
  J}^{0, 1} \,v_g^{\ast} \cdot^{} \xi \;\;=\;\; \left( \nabla_{g, J}^{0, 1} \,v_g^{\ast}
  \cdot^{} \xi \right)_g^T,\;\; \forall \,\xi\, \in\, T_X \Big\}\; .
\end{eqnarray*}
With this notation hold the following fundamental lemma which represents the key result of this section.

\begin{lemma}
  \label{Kah-crv}Let $(g_t)_{t \geqslant 0}$ be an arbitrary smooth family of
  Riemannian metrics and let $(J_t)_{t \geqslant 0}$ be a family of
  endomorphism sections of $T_X$ solution of the $\tmop{ODE}$
  \[ 2 \,\dot{J}_t \;\;=\;\; J_t \, \dot{g}_t^{\ast} \;\,-\;\,
      \dot{g}_t^{\ast} \, J_t\;, \]
  with K\"ahler initial data $(J_0, g_0)$. Then $(J_t, g_t)_{t \geqslant 0}$
  is a smooth family of K\"ahler structures if and only if $ \dot{g}_t \in
  \mathbbm{D}^{J_t}_{g_t}  \tmop{for} \tmop{all} t \geqslant 0$.
\end{lemma}

\begin{proof}
  We set $M_t \assign \nabla_{g_t} J_t$ and we expand the time derivative
  \begin{eqnarray*}
    2 \,\dot{M}_t & = & 2 \,\dot{\nabla}_{g_t} \,J_t \;\,+\;\, 2\, \nabla_{g_t}  \dot{J}_t\\
    &  & \\
    & = & 2 \,\dot{\nabla}_{g_t} \,J_t \;\,+\;\, \nabla_{g_t}\left( J_t \,
    \dot{g}_t^{\ast} \;\, - \;\, \dot{g}_t^{\ast} J_t
    \right)\\
    &  & \\
    & = & M_t  \,\dot{g}_t^{\ast} \;\, - \;\,
    \dot{g}_t^{\ast} \,M_t \;\,+\;\, 2\, \dot{\nabla}_{g_t} \,J_t
\\
\\
&+&
J_t\, \nabla_{g_t} 
    \,\dot{g}_t^{\ast} \;\, -\;\,  \nabla_{g_t} \,
    \dot{g}_t^{\ast} \,J_t \;.
  \end{eqnarray*}
  We observe now the identity
  \[ \big( \dot{\nabla}_{g_t} J_t \big) \eta \;\;=\;\; \dot{\nabla}_{g_t} \left(
     J_t \eta \right) \;\,-\;\, J_t\,  \dot{\nabla}_{g_t} \eta\;, \]
  for all $\eta \in T_X .$ Moreover the first variation formula for the
  Levi-Civita connection (\ref{Bess-var-LC-conn}) implies the identities
  \begin{eqnarray*}
    2 \,\dot{\nabla}_{g_t} \left( J_t \,\eta \right)  & = & \nabla_{g_t} \,
    \dot{g}_t^{\ast} \, J_t \,\eta \;\,+\;\, (J_t \,\eta) \;\neg\; \nabla_{g_t} 
    \,\dot{g}_t^{\ast} \;\,-\;\, \big( \nabla_{g_t} \, \dot{g}_t^{\ast} \, J_t\, \eta
    \big)_{g_t}^T\;,\\
    &  & \\
    - \;\,2\, J_t\,  \dot{\nabla}_{g_t} \eta & = & - \;\,J_t\, \nabla_{g_t}\, 
    \dot{g}_t^{\ast} \, \eta \;\,-\;\, \eta\; \neg\; J_t\, \nabla_{g_t}\,  \dot{g}_t^{\ast}
    \;\,+\;\, J_t  \left( \nabla_{g_t}  \,\dot{g}_t^{\ast} \, \eta \right)_{g_t}^T \;.
  \end{eqnarray*}
  We deduce the equalities
  \begin{eqnarray*}
    T_t \,\eta & \assign & \left( 2\, \dot{\nabla}_{g_t} \,J_t\;\, +\;\, J_t \nabla_{g_t}\, 
    \dot{g}_t^{\ast} \;\, -\;\, \nabla_{g_t} \,
    \dot{g}_t^{\ast} \,J_t \right) \eta\\
    &  & \\
    & = & 2 \,\dot{\nabla}_{g_t} \left( J_t \,\eta \right) \;\,-\;\, 2\, J_t \,
    \dot{\nabla}_{g_t} \,\eta 
\\
\\
&+& J_t\, \nabla_{g_t}\,  \dot{g}_t^{\ast}
    \;\, - \;\, \nabla_{g_t}  \,\dot{g}_t^{\ast} \,J_t\\
    &  & \\
    & = & (J_t \,\eta) \;\neg\; \nabla_{g_t} \, \dot{g}_t^{\ast} \;\,- \;\,\left( \nabla_{g_t}\, 
    \dot{g}_t^{\ast} \, J_t\, \eta \right)_{g_t}^T 
\\
\\
&-& \eta \;\neg\; J_t\,
    \nabla_{g_t}\,  \dot{g}_t^{\ast} \;\,+\;\, J_t  \left( \nabla_{g_t}\, 
    \dot{g}_t^{\ast} \, \eta \right)_{g_t}^T \;.
  \end{eqnarray*}
  Multiplying on the l.h.s with $g_t$ we infer for all $\xi, \mu \in T_X$ the
  expression
  \begin{eqnarray*}
     g_t \big( T_t (\xi\,, \eta)\,, \mu \big) & = &  g_t \big(
    \nabla_{g_t}  \,\dot{g}_t^{\ast} (J_t \eta\,, \xi)\,, \mu \big) \;\,-\;\, g_t \big( \xi\,,
    \nabla_{g_t}  \,\dot{g}_t^{\ast} (\mu\,, J_t \eta) \big)\\
    &  & \\
    & - & g_t \big( J_t \nabla_{g_t}\,  \dot{g}_t^{\ast} (\eta\,, \xi)\,, \mu
    \big) \;\,+\;\, g_t \big( J_t  \left( \nabla_{g_t}  \dot{g}_t^{\ast} \cdot \eta
    \right)_{g_t}^T \cdot \xi\,, \mu\big)\\
    &  & \\
    & = & g_t \big( \nabla_{g_t}\,  \dot{g}_t^{\ast} (J_t \eta\,, \xi)\,,
    \mu \big) \;\,-\;\, g_t \big( \nabla_{g_t}\,  \dot{g}_t^{\ast} (\mu\,, J_t \eta)\,, \xi
    \big)\\
    &  & \\
    & + & g_t \big( \nabla_{g_t}  \dot{g}_t^{\ast} (\eta\,, \xi)\,, J_t \mu
    \big)\;\, -\;\, g_t \big( \xi\,, \nabla_{g_t} \, \dot{g}_t^{\ast}  ( J_t \mu\,,
    \eta )\big)\\
    &  & \\
    & = & \nabla_{g_t} \, \dot{g}_t (J_t \eta\,, \xi\,, \mu) \;\,-\;\, \nabla_{g_t}\, 
    \dot{g}_t (\mu\,, J_t \eta\,, \xi)\\
    &  & \\
    & + & \nabla_{g_t} \, \dot{g}_t  ( \eta\,, \xi\,, J_t \mu ) \;\,-\;\,
    \nabla_{g_t} \, \dot{g}_t  ( J_t \mu\,, \eta\,, \xi )\\
    &  & \\
    & = & \nabla_{g_t}\,  \dot{g}_t (J_t \eta\,, \xi\,, \mu) \;\,-\;\, \nabla_{g_t}\, 
    \dot{g}_t (\mu\,, \xi\,, J_t \eta)\\
    &  & \\
    & + & \nabla_{g_t} \, \dot{g}_t  ( \eta\,, \xi\,, J_t \mu )\;\, -\;\,
    \nabla_{g_t} \, \dot{g}_t  ( J_t \mu\,, \xi\,, \eta )\\
    &  & \\
    & = & 2\, g_t \Big( \nabla^{0, 1}_{g_t, J_t} \, \dot{g}^{\ast}_t (\eta\,,
    \xi)\,, J_t \mu\Big) \;\,-\;\, 2 \,g_t \left( \nabla^{0, 1}_{g_t, J_t} \, \dot{g}^{\ast}_t
    (J_t \mu\,, \xi)\,, \eta \right)\;,
  \end{eqnarray*}
  thanks to obvious symmetries of $\nabla_{g_t}  \,\dot{g}_t$. Thus the
  equation
  \[ 2 \,\dot{M}_t \;\;=\;\; M_t\,  \dot{g}_t^{\ast} \;\, -\;\,
     \dot{g}_t^{\ast}\, M_t \;\,+\;\, T_t\;, \]
  with initial condition $M_0 \;=\; 0$ (thanks to the assumption) shows that $M_t
  \equiv 0$ if and only if $T_t \;\equiv\; 0$, i.e. if and only if the endomorphism
$$
\nabla^{0, 1}_{g_t, J_t}  \,\dot{g}^{\ast}_t \cdot \xi\;,
$$ 
is
  $g$-symmetric for all $\xi \in T_X $and all $t \geqslant 0$.
\end{proof}
We deduce the following remarkable corollary.
\begin{corollary}
 \label{F-Kah}Let $(g_t)_{t \geqslant 0}$ be a smooth family of Riemannian
  metrics such that $\dot{g}_t \in \mathbbm{F}_{g_t}$ and let $(J_t)_{t
  \geqslant 0}$ be a family of endomorphism sections of $T_X$ solution of the
  $\tmop{ODE}$
  \[ 2 \,\dot{J}_t \;\;=\;\; J_t\,  \dot{g}_t^{\ast} \;\, - \;\,
     \dot{g}_t^{\ast} J_t\;, \]
  with K\"ahler initial data $(J_0, g_0)$. Then $(J_t, g_t)_{t \geqslant 0}$
  is a smooth family of K\"ahler structures.
\end{corollary}

\begin{proof}
  We show that $\dot{g}_t \in \mathbbm{D}^{J_t}_{g_t}$ for all $t \geqslant 0$. 
For this purpose consider the 2-tensor $N_t^{\xi}$ defined by the formula
  \begin{eqnarray*}
    2 \,N^{\xi}_t (u, v) & \assign & 2 \,g_t \left( \nabla_{g_t, J_t}^{1, 0} \,
    \dot{g}_t^{\ast} (u\,, \xi)\,, v \right) \\
    &  & \\
    & = & g_t \Big( \nabla_{g_t} \, \dot{g}_t^{\ast} (u\,, \xi)\,, v
    \Big) \;\,-\;\, g_t \Big( \nabla_{g_t} \, \dot{g}_t^{\ast} (J_t u\,, \xi)\,, J_t \,v \Big)
    \\
    &  & \\
    & = & \nabla_{g_t} \, \dot{g}_t \,(u\,, \xi\,, v) \;\,-\;\, \nabla_{g_t}\,  \dot{g}_t \,(J_t
    u\,, \xi\,, J_t \,v) \;.
  \end{eqnarray*}
  This last expression shows that $N_t^{\xi}$ is symmetric for all $t
  \geqslant 0$ and all $\xi \in T_X$. In fact the 3-tensor $\nabla_{g_t} 
  \dot{g}_t$ is symmetric for all $t \geqslant 0$ since $\dot{g}_t \in
  \mathbbm{F}_{g_t}$ by assumption. \ 
\end{proof}

We give now a better description of the vector space $\mathbbm{D}_g^J$ in the
case $(J, g) \in \mathcal{K}\mathcal{S}.$ We start with the following quite
elementary fact.

\begin{lemma}
  \label{dec-sm-cx}Let $(J, g) \in \mathcal{K}\mathcal{S}$. Then for all $v
  \in \mathbbm{D}^J_g$ and $\xi \in T_X$ the endomorphisms
$$\nabla^{0, 1}_{g, J} \hspace{0.25em}
  (v_g^{\ast})_{_J}^{1, 0} \cdot \xi\;,
$$ 
and 
$$\nabla^{0, 1}_{g, J}
  \hspace{0.25em} (v_g^{\ast})_{_J}^{0, 1} \cdot \xi\;,
$$ 
are $g$-symmetric.
\end{lemma}

\begin{proof}
  We observe first the decomposition formula
  \begin{eqnarray*}
    2 \,\nabla^{0, 1}_{g, J} \hspace{0.25em} (v_g^{\ast})_{_J}^{1, 0} \cdot \xi
    & = & \nabla^{0, 1}_{g, J} \hspace{0.25em} v_g^{\ast} \cdot \xi \;\,-\;\, J\,
    \nabla^{0, 1}_{g, J} \hspace{0.25em} v_g^{\ast} \cdot J \,\xi \;.
  \end{eqnarray*}
  We observe also that the last endomorphism on the r.h.s is $g$-symmetric. In
  fact
  \begin{eqnarray*}
    (J \,\nabla^{0, 1}_{g, J} \hspace{0.25em} v_g^{\ast} \cdot J \,\xi)_g^T & = &
    - \;\,\nabla^{0, 1}_{g, J} \hspace{0.25em} v_g^{\ast} \,(J \cdot\,, J\, \xi) 
\\
\\
&=& J\,
    \nabla^{0, 1}_{g, J} \hspace{0.25em} v_g^{\ast} \cdot J\, \xi \;.
  \end{eqnarray*}
  We infer that the endomorphism 
$$
\nabla^{0, 1}_{g, J} \hspace{0.25em}
  (v_g^{\ast})_{_J}^{1, 0} \cdot \xi\;,
$$ 
is also $g$-symmetric. In a similar way
  we deduce that the endomorphism 
$$
\nabla^{0, 1}_{g, J} \hspace{0.25em}
  (v_g^{\ast})_{_J}^{0, 1} \cdot \xi\;,
$$ 
is also $g$-symmetric. 
\end{proof}

We compute now the tangent bundle to the space of (integrable) complex
structures. This result is well known. We include it here for the sake of
completeness.

\begin{lemma}
  \label{Tg-cx-str}Let $(J_t)_t \subset \mathcal{J}$ be a smooth family of
  complex structures. Then hold the identity 
$$
\overline{\partial}_{_{T_{X, J_t}}} \dot{J}_t
  \;\;\equiv\;\; 0\; .
$$ 
Moreover any smooth path $(J_t, g_t)_t \subset
  \mathcal{K}\mathcal{S}$ satisfies the property 
$$
\nabla_{g_t, J_t}^{0, 1}  \dot{J}_t \,\in\,
  S_{_{\mathbbm{R}}}^2 T^{\ast}_X \otimes T_X \;.
$$ 
\end{lemma}

\begin{proof}
  We write first the expression
  \begin{eqnarray*}
    \overline{\partial}_{_{T^{1, 0}_{X, J_t}}} \dot{J}_t  \left(
    \xi_{_{J_t}}^{0, 1}, \eta_{_{J_t}}^{0, 1} \right) & = & \left[
    \xi_{_{J_t}}^{0, 1}, \dot{J}_t \,\eta_{_{J_t}}^{0, 1} \right]_{_{J_t}}^{1,
    0} 
\\
\\
&-& \left[ \eta_{_{J_t}}^{0, 1}, \dot{J}_t \,\xi_{_{J_t}}^{0, 1}
    \right]_{_{J_t}}^{1, 0} \;\,-\;\, \dot{J}_t  \left[ \xi_{_{J_t}}^{0, 1},
    \eta_{_{J_t}}^{0, 1} \right]_{_{J_t}}^{0, 1},
  \end{eqnarray*}
  for all $\xi, \eta \in C^{\infty} (U, T_X)$ over an arbitrary open set $U \subset X$.
  Time deriving the integrability condition
  \begin{eqnarray*}
    \left[ \xi_{_{J_t}}^{0, 1}, \eta_{_{J_t}}^{0, 1} \right]_{_{J_t}}^{1, 0} &
    \equiv & 0\;,
  \end{eqnarray*}
  we obtain the equalities
  \begin{eqnarray*}
    0 & = & - \;\,\dot{J}_t  \left[ \xi_{_{J_t}}^{0, 1}, \eta_{_{J_t}}^{0, 1}
    \right] \;\,+\;\, \left[ \dot{J}_t \,\xi\,, \eta_{_{J_t}}^{0, 1} \right]_{_{J_t}}^{1,
    0} \;\,+\;\, \left[ \xi_{_{J_t}}^{0, 1}, \dot{J}_t \,\eta \right]_{_{J_t}}^{1, 0}\\
    &  & \\
    & = & - \;\,\dot{J}_t  \left[ \xi_{_{J_t}}^{0, 1}, \eta_{_{J_t}}^{0, 1}
    \right]_{_{J_t}}^{0, 1} \;\,-\;\, \left[ \eta_{_{J_t}}^{0, 1}, \big( \dot{J}_t\,
    \xi \big)_{_{J_t}}^{1, 0} \right]_{_{J_t}}^{1, 0} \;\,+\;\, \left[
    \xi_{_{J_t}}^{0, 1}, \big( \dot{J}_t \,\eta \big)_{_{J_t}}^{1, 0}
    \right]_{_{J_t}}^{1, 0},
  \end{eqnarray*}
  thanks again to the integrability condition of the complex structure $J_t$. Then the identity
$$
\overline{\partial}_{_{T^{1, 0}_{X, J_t}}} \dot{J}_t \;\;\equiv\;\; 0\;,
$$ 
follows from
  the equality 
  \begin{eqnarray*}
\big( \dot{J}_t \,\xi \big)_{_{J_t}}^{1, 0} & = & \dot{J}_t\,\xi_{_{J_t}}^{0, 1}\;,
  \end{eqnarray*}
  thanks to the $J_t$-anti-linearity of the endomorphism $\dot{J}_t$. Thus for any smooth path $(J_t, g_t)_t \subset
  \mathcal{K}\mathcal{S}$, the identity
\begin{eqnarray*}
    0 \;\;\equiv \;\;\overline{\partial}_{_{T_{X, J_t}}}  \dot{J}_t \,(\xi\,, \eta) & = &
    \nabla_{g_t, J_t}^{0, 1}  \dot{J}_t \,(\xi\,, \eta) \;\,-\;\, \nabla_{g_t, J_t}^{0, 1}
    \dot{J}_t \,(\eta\,, \xi)\;,
  \end{eqnarray*}
  implies the required conclusion.
\end{proof}
The following characterization will be quite crucial for the proof of the main theorem.
\begin{lemma}
  \label{rmq-kh-sm}Let $(J, g) \in \mathcal{K}\mathcal{S}$ and let $A \in
  C^{\infty} (X, T^{\ast}_{X, - J} \otimes_{_{\mathbbm{C}}} T_{X, J})$ be
  $g$-symmetric. Then the conditions
\begin{eqnarray}
\label{cx-sm-A3} \nabla_{g, J}^{0, 1}\, A \cdot \xi 
&=&
 \left( \nabla_{g,
    J}^{0, 1} \,A \cdot \xi \right)_g^T\;, 
\\\nonumber
\\
\label{cx-sm-A4}  \nabla_{g, J}^{0, 1} \,A 
&\in&
S_{_{\mathbbm{R}}}^2
    T^{\ast}_X \otimes T_X \;, 
\\\nonumber
\\
\label{cx-sm-A5}  \overline{\partial}_{_{T_{X, J}}} A 
&\equiv& 0\;,
\end{eqnarray}
  are equivalent. Moreover if $B \in C^{\infty} (X, T^{\ast}_{X, J}
  \otimes_{_{\mathbbm{C}}} T_{X, J})$ is $g$-symmetric then the conditions
\begin{eqnarray}
\label{Kh-sm-B2} \nabla_{g, J}^{0, 1} \,B \cdot \xi 
&=&
 \left( \nabla_{g,
    J}^{0, 1} \,B \cdot \xi \right)_g^T\;,
\\\nonumber
\\
\label{Kh-sm-B3} \nabla_{g, J}^{1, 0} \,B 
&\in&
S_{_{\mathbbm{R}}}^2
    T^{\ast}_X \otimes T_X\;, 
\\\nonumber
\\
\label{Kh-sm-B4} \partial^g_{_{T_{X, J}}} B 
&\equiv& 0\;,
\end{eqnarray}
  are equivalent.
\end{lemma}
\begin{proof}
  We observe first that the identities
\begin{eqnarray}
\label{cx-sm-A1} \xi \;\neg \;\nabla_{g, J}^{1, 0} \,A 
&=&
 \left( \xi \;\neg\;
    \nabla_{g, J}^{1, 0} \,A \right)_g^T\;, 
\\\nonumber
\\
\label{cx-sm-A2} \xi \;\neg\; \nabla_{g, J}^{0, 1} \,A 
&=&
\left( \xi \;\neg\;
    \nabla_{g, J}^{0, 1} \,A \right)_g^T\;,
\end{eqnarray}
are direct consequence of the identity $A = A_g^T .$ Using (\ref{cx-sm-A2})
  we show that (\ref{cx-sm-A3}) implies (\ref{cx-sm-A4}). In fact for all
  $\xi, \eta, \mu \in T_X$ hold the equalities
  \begin{eqnarray*}
    g \left( \nabla_{g, J}^{0, 1} \,A (\xi\,, \eta)\,, \mu \right) & = & g \left(
    \eta\,, \nabla_{g, J}^{0, 1} \,A (\xi\,, \mu) \right)\\
    &  & \\
    & = & g \left( \nabla_{g, J}^{0, 1} \,A (\eta\,, \mu)\,, \xi \right)\\
    &  & \\
    & = & g \left( \mu\,, \nabla_{g, J}^{0, 1} \,A (\eta\,, \xi) \right) \;.
  \end{eqnarray*}
  We observe now that (\ref{cx-sm-A2}) and (\ref{cx-sm-A4}) imply directly
  (\ref{cx-sm-A3}). We observe also that the identity (\ref{cx-sm-A4}) is
  equivalent to (\ref{cx-sm-A5}) \ thanks to the formula
  \begin{eqnarray*}
    \overline{\partial}_{_{T_{X, J}}} A \,(\xi\,, \eta) & = & \nabla_{g, J}^{0, 1}
    \,A \,(\xi\,, \eta) \;\,-\;\, \nabla_{g, J}^{0, 1} \,A\, (\eta\,, \xi) \;.
  \end{eqnarray*}
  We show now the statement concerning the $J$-linear endomorphism $B .$ We
  observe first that the identity
  \begin{equation}
    \label{Kh-sm-B1} \xi \;\neg\; \nabla_{g, J}^{1, 0} \,B \;\;=\;\; \left( \xi \;\neg\;
    \nabla_{g, J}^{0, 1}\, B \right)_g^T\;,
  \end{equation}
  follows immediately from the symmetry identity $B = B_g^T$. We show that
  (\ref{Kh-sm-B3}) follows from (\ref{Kh-sm-B1}) and (\ref{Kh-sm-B2}). In fact
  for all $\xi, \eta, \mu \in T_X$ hold the equalities
  \begin{eqnarray*}
    g \left( \nabla_{g, J}^{1, 0} \,B (\xi\,, \eta)\,, \mu \right) & = & g \left(
    \eta\,, \nabla_{g, J}^{0, 1} \,B (\xi\,, \mu) \right)\\
    &  & \\
    & = & g \left( \nabla_{g, J}^{0, 1} \,B (\eta\,, \mu)\,, \xi \right)\\
    &  & \\
    & = & g \left( \mu\,, \nabla_{g, J}^{1, 0} \,B (\eta\,, \xi) \right)\;,
  \end{eqnarray*}
which is equivalent to the identity
  (\ref{Kh-sm-B3}). We show now that (\ref{Kh-sm-B3}) implies
  (\ref{Kh-sm-B2}). In fact
  \begin{eqnarray*}
    g \left( \nabla_{g, J}^{0, 1} \,B (\eta\,, \xi)\,, \mu \right) & = & g \left(
    \xi\,, \nabla_{g, J}^{1, 0} \,B (\eta\,, \mu) \right)\\
    &  & \\
    & = & g \left( \xi\,, \nabla_{g, J}^{1, 0} \,B (\mu\,, \eta) \right)\\
    &  & \\
    & = & g \left( \nabla_{g, J}^{0, 1} \,B (\mu\,, \xi)\,, \eta \right) \;.
  \end{eqnarray*}
  Finally we observe that the identity (\ref{Kh-sm-B3}) is equivalent to
  (\ref{Kh-sm-B4}) thanks to the formula
  \begin{eqnarray*}
    \partial^g_{_{T_{X, J}}} B \,(\xi\,, \eta) & = & \nabla_{g, J}^{1, 0} \,B \,(\xi\,,
    \eta) \;\,-\;\, \nabla_{g, J}^{1, 0}\, B (\eta\,, \xi)\; .
  \end{eqnarray*}
\end{proof}
We observe in conclusion that for any K\"ahler structure $(J, g)$ hold the
equalities
\begin{equation}
  \label{Kah-D} \mathbbm{D}^J_g \;=\; \Big\{ v \in C^{\infty} \left( X,
  S_{_{\mathbbm{R}}}^2 T^{\ast}_X \right) \mid \hspace{0.25em}
  \partial^g_{_{T_{X, J}}} (v_g^{\ast})_{_J}^{1, 0} \;=\; 0\;,\;
  \overline{\partial}_{_{T_{X, J}}} (v_g^{\ast})_{_J}^{0, 1} \;= \;0 \Big\}\; .
\end{equation}
and
\begin{equation}
  \label{Kah-F} \mathbbm{F}_g \;=\; \Big\{ v \in \mathbbm{D}^J_g \mid \hspace{0.25em} 
  \overline{\partial}_{_{T_{X, J}}} (v_g^{\ast})_{_J}^{1, 0} \;=\; -\;\,
  \partial^g_{_{T_{X, J}}} (v_g^{\ast})_{_J}^{0, 1}  \Big\}\; ,
\end{equation}
for bi-degree reasons.
Comparing the previous $(1, 1)$-forms by means of $g$-geodesic and
$J$-holomorphic coordinates we infer also the identity
\begin{equation}
  \label{sup-Kh-sm} \mathbbm{F}_g \;=\; \Big\{ v \in \mathbbm{D}^J_g \mid
  \hspace{0.25em} \xi \;\neg\; \nabla_{g, J}^{0, 1} (v_g^{\ast})_{_J}^{1, 0} \;=\;
  \nabla_{g, J}^{1, 0} (v_g^{\ast})_{_J}^{0, 1} \cdot \xi\,,\; \forall \xi \in T_X
  \Big\}\; .
\end{equation}
We denote by $\mathcal{K}_{_{^J}} \subset \mathcal{M}$ the closed subset of $J$-invariant K\"ahler metrics. Its tangent space at a point $g \in
\mathcal{K}_{_{^J}}$ is
\begin{eqnarray*}
  \hat{\mathbbm{D}}^J_g & = & \Big\{ v \in C^{\infty} \left( X,
  S_{_{\mathbbm{R}}}^2 T^{\ast}_X \right) \mid \hspace{0.25em} v \;=\; J^{\ast} v\,
  J\,,\;\; d (v \,J)\; =\; 0 \Big\}\; .
\end{eqnarray*}
We observe also that
\begin{eqnarray*}
  \hat{\mathbbm{D}}^J_g & = & \Big\{ v \in C^{\infty} \left( X,
  S_{_{\mathbbm{R}}}^2 T^{\ast}_X \right) \mid \hspace{0.25em} v \;=\; J^{\ast} v\,
  J\,,\;\; \partial^g_{_{T_{X, J}}} v_g^{\ast} \;=\; 0 \Big\} \;.
\end{eqnarray*}
In fact the condition $d (v \,J)\; = \;0$ is equivalent to the condition
$\partial_{_J} (v \,J)\; =\; 0$, which in its turn is equivalent to the condition
$$
\partial^g_{_{T_{X, J}}} v_g^{\ast} \;\;=\;\; 0\;.
$$
Moreover hold the equalities
\begin{eqnarray*}
  \hat{\mathbbm{D}}^J_g \cap \mathbbm{F}_g & = & \Big\{ v \in C^{\infty} \left( X,
  S_{_{\mathbbm{R}}}^2 T^{\ast}_X \right) \mid \hspace{0.25em} v \,=\, J^{\ast} v\,
  J\,,\, \partial^g_{_{T_{X, J}}} v_g^{\ast} \,=\, 0\,,\, \overline{\partial}_{_{T_{X,
  J}}} v_g^{\ast} \,=\, 0  \Big\} \\
  &  & \\
  & = & \Big\{ v \in C^{\infty} \left( X, S_{_{\mathbbm{R}}}^2 T^{\ast}_X \right)
  \mid \hspace{0.25em} v \;=\; J^{\ast} v\, J\;, \nabla_g \,v \;=\; 0 \Big\}\;,
\end{eqnarray*}
thanks to the $\partial \overline{\partial}$-lemma applied to the $d$-closed
$(1, 1)$-form $v J .$ We notice that the last identity is also reflected
from a comparison of lemma \ref{Kah-vr-W-fxcx} with corollary
\ref{cor-II-Var-W}.
\section{The total second variation of $\mathcal{W}$ along K\"ahler
structures}

We are now in position to show our main result, theorem \ref{lm-Kah-IIvr-W}.

\begin{proof}
  Our proof is based on the decomposition of the general second variation
  formula (\ref{rm-2vr-W}) via the symmetries explained in lemma \ref{rmq-kh-sm}, which hold 
thanks to lemmas \ref{Kah-crv} and \ref{dec-sm-cx} in the previous section. For notation simplicity we set 
$$
A \;\;\assign\;\; -\;\, (v_{_J}'')^{\ast} \;\;=\;\; -\;\,
  (v_g^{\ast})_{_J}^{0, 1}\;,
$$ 
and 
$$
B \;\;\assign\;\; (v_{_J}')^{\ast} \;\;=\;\;
  (v^{\ast}_g)_{_J}^{1, 0}\,.
$$ 
We decompose first the squared norm
  \begin{eqnarray*}
    | \nabla_g \,v|^2_g & = & | \nabla_g \,v'_{_J} |^2_g \;\,+\;\, | \nabla_g\, v''_{_J}
    |^2_g\\
    &  & \\
    & = & | \nabla_g \,B|^2_g \;\,+\;\, | \nabla_g \,A|^2_g\; .
  \end{eqnarray*}
We observe in fact that for any $g$-orthonormal (real) basis $(e_k)^{2 n}_{k = 1}$ hold
  \begin{eqnarray*}
    \big\langle \nabla_g v'_{_J}\,, \nabla_g v''_{_J}  \big\rangle_g & = &
    \big\langle \nabla_{g, e_k} v'_{_J}\,, \nabla_{g, e_k} v''_{_J} 
    \big	\rangle_g\\
    &  & \\
    & = &  -\;\, \tmop{Tr}_{_{\mathbbm{R}}}  \Big[ \nabla_{g, e_k} \,B\; \nabla_{g,
    e_k} \,A \Big]\\
    &  & \\
    & = & 0 \;.
  \end{eqnarray*}
  This is because the endomorphisms $\nabla_{g, e_k} B$ and $\nabla_{g, e_k}
  A$ are respectively $J$-linear and $J$-anti linear. We decompose now the
  squared norm
  \begin{eqnarray*}
    \frac{1}{6} \hspace{0.25em} \big| \hat{\nabla}_g \,v\big|^2_g & = & 
    \frac{1}{6} \hspace{0.25em} \big| \hat{\nabla}_g \,v'_{_J} \big|^2_g 
    \;\,+ \;\, \frac{1}{3}\hspace{0.25em}\big  \langle \hat{\nabla}_g \,v'_{_J}\,,
    \hat{\nabla}_g \,v''_{_J}  \big\rangle_g \;\,+\;\, \frac{1}{6} \hspace{0.25em} \big|
    \hat{\nabla}_g \,v_{_J}'' \big|^2_g\; .
  \end{eqnarray*}
  We explicit first the therm in the middle of the r.h.s since is the most
  complicated one. Using the identity
  \begin{eqnarray*}
    0 & = & \big\langle \nabla_g \,v'_{_J}\,, \nabla_g \,v''_{_J} 
    \big\rangle_g 
\;\;=\;\; \nabla_g \,v'_{_J} (e_j, e_k, e_l) \,\nabla_g \,v''_{_J} (e_j,
    e_k, e_l),
  \end{eqnarray*}
  and the symmetry in the last two entries $k, l,$ we expand the product therm
  \begin{eqnarray*}
    \big\langle \hat{\nabla}_g \,v'_{_J}, \hat{\nabla}_g v''_{_J} 
    \big\rangle_g & = & \Big[ \nabla_g \,v'_{_J} (e_j, e_k, e_l) \,+\, \nabla_g\,
    v'_{_J} (e_k, e_j, e_l) \,+\, \nabla_g \,v'_{_J} (e_l, e_j, e_k)\Big] \times\\
    &  & \\
    & \times & \Big[ \nabla_g \,v''_{_J} (e_j, e_k, e_l) \,+\, \nabla_g\, v_{_J}'' (e_k,
    e_j, e_l) \,+\, \nabla_g\, v''_{_J} (e_l, e_j, e_k)\Big]\\
    &  & \\
    & = & 6 \,\nabla_g \,v'_{_J} (e_k, e_j, e_l)\, \nabla_g \,v''_{_J} (e_l, e_j,
    e_k)\\
    &  & \\
    & = & - \;\,6\, g \nabla_g\, B \,(e_k, e_j, e_l)\; g \nabla_g \,A \,(e_l, e_j, e_k)\\
    &  & \\
    & = & - \;\,6\, g \nabla^{1, 0}_{g, J}\, B\, (e_k, e_j, e_l)\; g \nabla^{1, 0}_{g, J}\,
    A \,(e_l, e_j, e_k)\\
    &  & \\
    & - & 6\, g \nabla^{0, 1}_{g, J} \,B \,(e_k, e_j, e_l) \;g \nabla^{1, 0}_{g, J} \,A\,
    (e_l, e_j, e_k)\\
    &  & \\
    & - & 6 \,g \nabla^{1, 0}_{g, J}\, B \,(e_k, e_j, e_l)\; g \nabla^{0, 1}_{g, J} \,A\,
    (e_l, e_j, e_k)\\
    &  & \\
    & - & 6 \,g \nabla^{0, 1}_{g, J} \,B\, (e_k, e_j, e_l) \;g \nabla^{0, 1}_{g, J} \,A\,
    (e_l, e_j, e_k)\\
    &  & \\
    & = & - \;\,6\, g \nabla^{1, 0}_{g, J}\, B\, (e_k, e_j, e_l) \;g \nabla^{1, 0}_{g, J}\,
    A\, (e_l, e_k, e_j)\qquad \tmop{by}\; ( \ref{cx-sm-A1})\\
    &  & \\
    & - & 6 \,\tmop{Tr}_{_{\mathbbm{R}}} \left[ \nabla^{1, 0}_{g, J} \,A \left(
    \nabla^{0, 1}_{g, J} \,B (\ast, e_j), e_j \right) \right]\\
    &  & \\
    & - & 6 \,\tmop{Tr}_{_{\mathbbm{R}}} \left[ \nabla^{0, 1}_{g, J}\, A \left(
    \nabla^{1, 0}_{g, J} \,B (\ast, e_j), e_j \right) \right]\\
    &  & \\
    & - & 6\, g \nabla^{0, 1}_{g, J} \,B \,(e_l, e_j, e_k)\; g \nabla^{0, 1}_{g, J} \,A
    \,(e_l, e_k, e_j) \quad\tmop{by} \;( \ref{Kh-sm-B2}) \;\tmop{and}\; ( \ref{cx-sm-A2})\\
    &  & \\
    & = & - \;\,6\, g \left( \nabla^{1, 0}_{e_k} \,B\, \nabla^{1, 0}_{e_j} \,A\, e_k\,, e_j
    \right)\\
    &  & \\
    & - & 6 \,\tmop{Tr}_{_{\mathbbm{R}}} \left[ \nabla^{0, 1}_{e_l} \,B\,
    \nabla^{0, 1}_{e_j} \,A \right] \\
    &  & \\
    & = & - \;\,6\, g \left( \nabla^{1, 0}_{e_p} \,B\, \nabla^{1, 0}_{e_r} \,A\, e_p\,, e_r
    \right) \;.
  \end{eqnarray*}
  Let $(\eta_k)^n_{k = 1} \subset T_{X, J, x_0}$ be a $g J$-orthonormal
  complex basis at an arbitrary point $x_0\in
X$ and let $(e_k)^{2 n}_{k = 1} \;=\; (\eta_k, J \eta_k)^n_{k = 1}$.
  Then hold the identity
  \begin{eqnarray*}
    \big \langle \hat{\nabla}_g v'_{_J}\,, \hat{\nabla}_g v''_{_J} 
    \big\rangle_g & = & - \;\,6 \cdot 4 \,g \left( \nabla^{1, 0}_{\eta_p} \,B\;
    \nabla^{1, 0}_{\eta_r} \,A \,\eta_p\,, \eta_r \right) \;.
  \end{eqnarray*}
  We consider now $g$-geodesic and $J$-holomorphic coordinates centered at the
  point $x_0$ and let $\zeta_k \assign \frac{\partial}{\partial z_k}$, $\eta_k
  \assign \zeta_k + \bar{\zeta}_k$. Let also 
$$
A \;\;=\;\; i\, A_{k, \bar{l}} \;
  \bar{\zeta}_l^{\ast} \otimes \zeta_k \;\,+\;\, \tmop{Conjugate}\;,
$$ 
and 
$$B \;\;=\;\; B_{k,
  \bar{l}} \;\zeta^{\ast}_k \otimes \zeta_l \;\,+\;\, \tmop{Conjugate}\;,
$$ 
be the local
  expressions of $A$ and $B$ with respect to this coordinates. By using the
  local expressions
\begin{eqnarray*}
    \nabla^{1, 0}_{\eta_r} \,A & = & i \,\partial_r \,A_{k, \bar{l}} \;
    \bar{\zeta}^{\ast}_l \otimes \zeta_k \;\,+\;\, \tmop{Conjugate}\;,\\
    &  & \\
\nabla^{1, 0}_{\eta_p} \,B & = & \partial_p \,B_{k, \bar{l}} \;\zeta^{\ast}_k
    \otimes \zeta_l \;\,+\;\, \tmop{Conjugate}\;,
  \end{eqnarray*}
  at the point $x_0$ we infer the identity
  \begin{eqnarray*}
    \nabla^{1, 0}_{\eta_p} \,B \;\nabla^{1, 0}_{\eta_r} \,A\, \eta_p & = & i\,
    \partial_r \,A_{k, \bar{p}} \;\partial_p\, B_{k, \bar{l}} \;\zeta_l \;\,+\;\,
    \tmop{Conjugate}\;,
  \end{eqnarray*}
  and thus
  \begin{eqnarray*}
    \big\langle \hat{\nabla}_g \,v'_{_J}\,, \hat{\nabla}_g \,v''_{_J} 
    \big\rangle_g & = &  -\;\, 6 \cdot 4\,\mathfrak{R}e \Big[ i\,
    \partial_r \,A_{k, \bar{p}} \;\partial_p \,B_{k, \bar{r}} \Big]\;,
  \end{eqnarray*}
  at the point $x_0 .$ Using this last equality we will show the following
  fundamental identity
  \begin{equation}
    \label{orth-der}  \big\langle \hat{\nabla}_g \,v'_{_J}\,,
    \hat{\nabla}_g \,v''_{_J}  \big\rangle_g \;\;=\;\; 6\, \left\langle \partial^{g
    }_{_{T_{X, J}}} A\,, \overline{\partial}_{_{T_{X, J}}} B \right\rangle_g \;.
  \end{equation}
  For this purpose we consider the decomposition of the scalar product at
  the point $x_0$
  \begin{eqnarray*}
    \left \langle \partial^g_{_{T_{X, J}}} A\,, \overline{\partial}_{_{T_{X,
    J}}} B \right\rangle_g & = & \left \langle \eta_r \;\neg\;
    \partial^g_{_{T_{X, J}}} A\,, \eta_r \;\neg\; \overline{\partial}_{_{T_{X, J}}}
    B \right\rangle_g \\
    &  & \\
    & + & \left\langle J \eta_r \;\neg\; \partial^g_{_{T_{X, J}}} A\,, J \eta_r
    \;\neg\; \overline{\partial}_{_{T_{X, J}}} B \right\rangle_g\\
    &  & \\
    & = & \tmop{Tr}_{_{\mathbbm{R}}} \left[ \left( \eta_r \;\neg\;
    \partial^g_{_{T_{X, J}}} A \right) \left( \eta_r \;\neg\;
    \overline{\partial}_{_{T_{X, J}}} B \right)_g^T \right]\\
    &  & \\
    & + & \tmop{Tr}_{_{\mathbbm{R}}} \left[ \left( J \eta_r \;\neg\;
    \partial^g_{_{T_{X, J}}} A \right) \left( J \eta_r \;\neg\;
    \overline{\partial}_{_{T_{X, J}}} B \right)_g^T \right] \;.
  \end{eqnarray*}
  Moreover using the local expressions \
  \begin{eqnarray*}
    \partial^g_{_{T_{X, J}}} A & = & i\, \partial_p\, A_{k, \bar{l}}  \;\left(
    \zeta_p^{\ast} \wedge \bar{\zeta}_l^{\ast} \right) \otimes \zeta_k \;\,+\;\,
    \tmop{Conjugate}\;,\quad (\tmop{at} \tmop{the} \tmop{point} x_0)\\
    &  & \\
    \overline{\partial}_{_{T_{X, J}}} B & = & - \;\,\partial_{\bar{p}} \,B_{k,
    \bar{l}}  \;\left( \zeta_k^{\ast} \wedge \bar{\zeta}_p^{\ast} \right)
    \otimes \zeta_l \;\,+\;\, \tmop{Conjugate}\;,
  \end{eqnarray*}
  we infer the identities
  \begin{eqnarray*}
    \eta_r \;\neg\; \partial^g_{_{T_{X, J}}} A & = & i\, \partial_r \,A_{k, \bar{l}} \;
    \bar{\zeta}_l^{\ast} \otimes \zeta_k \;\,-\;\, i\, \partial_l \,A_{k, \bar{r}}\;
    \zeta^{\ast}_l \otimes \zeta_k \;\,+\;\, \tmop{Conjugate}\;,\\
    &  & \\
    J \eta_r \;\neg\; \partial^{g J}_{_{T_{X, J}}} A & = & - \;\,\partial_r \,A_{k,
    \bar{l}}  \;\bar{\zeta}_l^{\ast} \otimes \zeta_k \;\,-\;\, \partial_l \,A_{k, \bar{r}}\;
    \zeta^{\ast}_l \otimes \zeta_k \;\,+\;\, \tmop{Conjugate}\;,\\
    &  & \\
    \eta_r \;\neg\; \overline{\partial}_{_{T_{X, J}}} B & = & - \;\,\partial_{\bar{k}}\,
    B_{r, \bar{l}}  \;\bar{\zeta}_k^{\ast} \otimes \zeta_l \;\,+\;\, \partial_{\bar{r}}\,
    B_{k, \bar{l}} \;\zeta^{\ast}_k \otimes \zeta_l \;\,+\;\, \tmop{Conjugate}\;,\\
    &  & \\
    J \eta_r \;\neg\; \overline{\partial}_{_{T_{X, J}}} B & = & -\;\, i\,
    \partial_{\bar{k}} \,B_{r, \bar{l}} \; \bar{\zeta}_k^{\ast} \otimes \zeta_l \;\,-\;\,
    i\, \partial_{\bar{r}} \,B_{k, \bar{l}} \;\zeta^{\ast}_k \otimes \zeta_l \;\,+\;\,
    \tmop{Conjugate} \;.
  \end{eqnarray*}
  (By conjugate we mean the complex conjugate of all therms preceding this
  word.) The fact that the endomorphism $B$ is $g$-symmetric implies the
  equality 
$$
\overline{B}_{k, \bar{l}} \;\;=\;\; B_{l, \bar{k}} \;\,+\;\, O \left( |z|^2
  \right)\;.
$$ 
Thus hold the equalities
  \begin{eqnarray*}
    \left( \eta_r \;\neg\; \overline{\partial}_{_{T_{X, J}}} B \right)_g^T & = & -
    \partial_{\bar{l}} \,B_{r, \bar{k}} \; \bar{\zeta}_k^{\ast} \otimes \zeta_l \;\,+\;\,
    \partial_r  \,\overline{B}_{l, \bar{k}} \;\zeta^{\ast}_k \otimes \zeta_l \;\,+\;\,
    \tmop{Conjugate}\;,\\
    &  & \\
    & = & - \;\,\partial_{\bar{l}} \,B_{r, \bar{k}} \; \bar{\zeta}_k^{\ast} \otimes
    \zeta_l \;\,+\;\, \partial_r \,B_{k, \bar{l}} \;\zeta^{\ast}_k \otimes \zeta_l \;\,+\;\,
    \tmop{Conjugate}\;,\\
    &  & \\
    \left( J \eta_r \;\neg\; \overline{\partial}_{_{T_{X, J}}} B \right)_g^T & = &
    - \;\,i\, \partial_{\bar{l}} \,B_{r, \bar{k}} \; \bar{\zeta}_k^{\ast} \otimes
    \zeta_l \;\,+\;\, i \,\partial_r \, \overline{B}_{l, \bar{k}} \;\zeta^{\ast}_k \otimes
    \zeta_l \;\,+\;\, \tmop{Conjugate}\;,\\
    &  & \\
    & = & - \;\,i\, \partial_{\bar{l}} \,B_{r, \bar{k}} \; \bar{\zeta}_k^{\ast} \otimes
    \zeta_l \;\,+\;\, i\, \partial_r \,B_{k, \bar{l}} \;\zeta^{\ast}_k \otimes \zeta_l \;\,+\;\,
    \tmop{Conjugate}\;,
  \end{eqnarray*}
  at the point $x_0$. At this point we obtain the equalities
\begin{eqnarray*}
&&
    \tmop{Tr}_{_{\mathbbm{R}}} \left[ \left( \eta_r \;\neg\; \partial^g_{_{T_{X,
    J}}} A \right) \left( \eta_r \;\neg\; \overline{\partial}_{_{T_{X, J}}} B
    \right)_g^T \right] 
\\
\\
& = & - \;\,i \,\partial_l \,C_{k, \bar{r}} \;\partial_r \,B_{k,
    \bar{l}} \;\,+\;\, \tmop{Conjugate}\\
    &  & \\
    & = & \tmop{Tr}_{_{\mathbbm{R}}} \left[ \left( J \eta_r \;\neg\;
    \partial^g_{_{T_{X, J}}} A \right) \left( J \eta_r \;\neg\;
    \overline{\partial}_{_{T_{X, J}}} B \right)_g^T \right]\;,
  \end{eqnarray*}
  and thus
  \begin{eqnarray*}
    \left \langle \partial^g_{_{T_{X, J}}} A\,, \overline{\partial}_{_{T_{X,
    J}}} B \right\rangle_g \;\;=\;\; -\;\, 4\,\mathfrak{R}e \Big[ i \,\partial_l \,A_{k,
    \bar{r}} \;\partial_r \,B_{k, \bar{l}} \Big]\;,
  \end{eqnarray*}
  which implies the identity (\ref{orth-der}). We expand now the squared norm
  \begin{eqnarray*}
    \frac{1}{6} \, \big| \hat{\nabla}_g \,v'_{_J} \big|^2_g & = &
    \frac{1}{2} \,\big| \nabla_g\, v'_{_J} \big|^2_g \;\,+\;\, \nabla_g \,v'_{_J} (e_k, e_j, e_l)\;
    \nabla_g \,v'_{_J} (e_l, e_j, e_k)\\
    &  & \\
    & = & \frac{1}{2}\,\big | \nabla_g \,B\big|^2_g \;\,+\;\, g \nabla_g \,B \,(e_k, e_j, e_l)\; g
    \nabla_g \,B \,(e_l, e_j, e_k) \;.
  \end{eqnarray*}
  Expanding further the therm
  \begin{eqnarray*}
&&
    g \nabla_g \,B \,(e_k, e_j, e_l) \;g\nabla_g \,B \,(e_l, e_j, e_k) 
\\
\\
& = & g
    \nabla^{1, 0}_{g, J} \,B \,(e_k, e_j, e_l)\; g \nabla^{1, 0}_{g, J} \,B \,(e_l, e_j,
    e_k)\\
    &  & \\
    & + & g \nabla^{0, 1}_{g, J} \,B\, (e_k, e_j, e_l)\; g \nabla^{1, 0}_{g, J} \,B\,
    (e_l, e_j, e_k)\\
    &  & \\
    & + & g \nabla^{1, 0}_{g, J} \,B\, (e_k, e_j, e_l)\; g \nabla^{0, 1}_{g, J}\, B\,
    (e_l, e_j, e_k)\\
    &  & \\
    & + & g \nabla^{0, 1}_{g, J}\, B\, (e_k, e_j, e_l) \;g \nabla^{0, 1}_{g, J} \,B\,
    (e_l, e_j, e_k)\\
    &  & \\
    & = & g \nabla^{1, 0}_{g, J} \,B (e_j, e_k, e_l) \;g \nabla^{1, 0}_{g, J} \,B\,
    (e_j, e_l, e_k) \qquad\tmop{by}\quad ( \ref{Kh-sm-B3})\\
    &  & \\
    & + & \tmop{Tr}_{_{\mathbbm{R}}} \left[ \nabla^{1, 0}_{g, J} \,B \left(
    \nabla^{0, 1}_{g, J} \,B (\ast, e_j), e_j \right) \right]\\
    &  & \\
    & + & \tmop{Tr}_{_{\mathbbm{R}}} \left[ \nabla^{0, 1}_{g, J} \,B \left(
    \nabla^{1, 0}_{g, J} \,B (\ast, e_j), e_j \right) \right]\\
    &  & \\
    & + & g \nabla^{0, 1}_{g, J} \,B \,(e_l, e_j, e_k)\; g \nabla^{0, 1}_{g, J} \,B\,
    (e_l, e_j, e_k) \quad\tmop{by} \quad ( \ref{Kh-sm-B2})\\
    &  & \\
    & = & \tmop{Tr}_{_{\mathbbm{R}}} \left( \nabla^{1, 0}_{g, e_j} \,B
    \right)^2 \;\,+\;\, \big| \nabla^{0, 1}_{g, J} \,B\big|^2_g\\
    &  & \\
    & = & \big| \nabla^{0, 1}_{g, J} \,B\big|^2_g\; .
  \end{eqnarray*}
  We observe in particular that by this computation follows the equalities
  \begin{eqnarray*}
   \big| \nabla^{0, 1}_{g, J} \,B\big|^2_g & = & g \nabla^{0, 1}_{g, J} \,B\, (e_k, e_j,
    e_l) \;g \nabla^{0, 1}_{g, J} \,B\, (e_l, e_j, e_k)\\
    &  & \\
    & = & g \nabla^{1, 0}_{g, J} \,B\, (e_k, e_l, e_j)\; g \nabla^{0, 1}_{g, J} \,B\,
    (e_k, e_j, e_l) \quad \tmop{by}\quad ( \ref{Kh-sm-B1}) \quad\tmop{and} \quad(
    \ref{Kh-sm-B2})\\
    &  & \\
    & = & g \nabla^{1, 0}_{g, J} \,B\, (e_k, e_l, e_j)\; g \nabla^{1, 0}_{g, J} \,B\,
    (e_k, e_l, e_j) \quad\tmop{by}\quad ( \ref{Kh-sm-B1})\\
    &  & \\
    & = & \big| \nabla^{1, 0}_{g, J}\, B\big|^2_g \;.
  \end{eqnarray*}
  We expand now the norm squared
  \begin{eqnarray*}
    \frac{1}{6} \, \big| \hat{\nabla}_g \,v''_{_J} \big|^2_g & = &
    \frac{1}{2} \,\big| \nabla_g \,v''_{_J} \big|^2_g \;\,+\;\, \nabla_g \,v_{_J}'' (e_k, e_j, e_l)
    \;\nabla_g \,v''_{_J} (e_l, e_j, e_k)\\
    &  & \\
    & = & \frac{1}{2}\, \big| \nabla_g \,A\big|^2_g \;\,+\;\, g \nabla_g \,A \,(e_k, e_j, e_l)\; g
    \nabla_g \,A\, (e_l, e_j, e_k)\; .
  \end{eqnarray*}
  In a similar way as we did for the endomorphism section $B$ we expand the therm
  \begin{eqnarray*}
&&
    g \nabla_g \,A\, (e_k, e_j, e_l) \;g \nabla_g \,A\, (e_l, e_j, e_k) 
\\
\\
& = & g
    \nabla^{1, 0}_{g, J} \,A\, (e_k, e_j, e_l)\; g \nabla^{1, 0}_{g, J} \,A\, (e_l, e_j,
    e_k)\\
    &  & \\
    & + & g \nabla^{0, 1}_{g, J} \,A\, (e_k, e_j, e_l) \;g \nabla^{1, 0}_{g, J} \,A\,
    (e_l, e_j, e_k)\\
    &  & \\
    & + & g \nabla^{1, 0}_{g, J} \,A \,(e_k, e_j, e_l) \;g \nabla^{0, 1}_{g, J} \,A\,
    (e_l, e_j, e_k)\\
    &  & \\
    & + & g \nabla^{0, 1}_{g, J} \,A (e_k, e_j, e_l)\; g \nabla^{0, 1}_{g, J} \,A\,
    (e_l, e_j, e_k)\\
    &  & \\
    & = & \tmop{Tr}_{_{\mathbbm{R}}} \left[ \nabla^{1, 0}_{g, J} \,A \left(
    \nabla^{1, 0}_{g, J} \,A (\ast, e_j), e_j \right) \right]\\
    &  & \\
    & + & \tmop{Tr}_{_{\mathbbm{R}}} \left[ \nabla^{1, 0}_{g, J} \,A \left(
    \nabla^{0, 1}_{g, J} \,A (\ast, e_j), e_j \right) \right]\\
    &  & \\
    & + & \tmop{Tr}_{_{\mathbbm{R}}} \left[ \nabla^{0, 1}_{g, J} \,A \left(
    \nabla^{1, 0}_{g, J} \,A (\ast, e_j), e_j \right) \right]\\
    &  & \\
    & + & g \nabla^{0, 1}_{g, J} \,A \,(e_j, e_k, e_l) \;g \nabla^{0, 1}_{g, J} \,A\,
    (e_j, e_l, e_k)\quad \tmop{by} \quad( \ref{cx-sm-A4})\\
    &  & \\
    & = & \big| \nabla^{0, 1}_{g, J} \,A\big|^2_g \;.
  \end{eqnarray*}
  In conclusion we obtain the identity
  \begin{eqnarray*}
    \frac{1}{6} \,\big| \hat{\nabla}_g \,v\big|^2_g & = & \big |
    \nabla_g \,B\big|^2_g \;\,+\;\, 2 \,\left\langle \partial^g_{_{T_{X, J}}} A\,,
    \overline{\partial}_{_{T_{X, J}}} B \right\rangle_g
\\
\\
& +&
 \frac{1}{2}\,\big |
    \nabla_g \,A\big|^2_g \;\,+\;\, \big| \nabla^{0, 1}_{g, J} \,A\big|^2_g\;,
  \end{eqnarray*}
  and thus
  \begin{eqnarray*}
    \frac{1}{6} \,\big| \hat{\nabla}_g \,v\big|^2_g \;\,-\;\, \big| \nabla_g \,v\big|^2_g &
    = & 2\,\left \langle \partial^g_{_{T_{X, J}}} A\,,
    \overline{\partial}_{_{T_{X, J}}} B \right\rangle_g 
\\
\\
&-& \frac{1}{2}\,\big |
    \nabla_g \,A\big|^2_g \;\,+\;\, \big| \nabla^{0, 1}_{g, J} \,A\big|^2_g\\
    &  & \\
    & = & 2\left \langle \partial^g_{_{T_{X, J}}} A\,,
    \overline{\partial}_{_{T_{X, J}}} B \right\rangle_g 
\\
\\
&-&
 \frac{1}{2}\,\big |
    \nabla^{1, 0}_{g, J} \,A\big|^2_g \;\,+\;\, \frac{1}{2}\,\big | \nabla^{0, 1}_{g, J} \,A\big|^2_g \;.
  \end{eqnarray*}
  We observe now the identity
  \begin{equation}
    \label{der-A-del-A} \big| \nabla^{1, 0}_{g, J} \,A\big|^2_g \;\;=\;\; \big| \partial^g_{_{T_{X,
    J}}} A\big|_g^2 \;.
  \end{equation}
  In fact expanding the squared norm
  \begin{eqnarray*}
    \big| \nabla^{1, 0}_{g, J} \,A\big|^2_g & = & \tmop{Tr}_{_{\mathbbm{R}}} \big(
    \nabla^{1, 0}_{\eta_r} \,A \big)^2 \;\,+\;\, \tmop{Tr}_{_{\mathbbm{R}}} \big(
    \nabla^{1, 0}_{J \eta_r} \,A \big)^2\\
    &  & \\
    & = & 2 \tmop{Tr}_{_{\mathbbm{R}}} \big( \nabla^{1, 0}_{\eta_r} \,A
    \big)^2\;,
  \end{eqnarray*}
  and using the local expression
  \begin{eqnarray*}
    \nabla^{1, 0}_{\eta_r} \,A & = & i\, \partial_r A_{k, \bar{l}} \;
    \bar{\zeta}^{\ast}_l \otimes \zeta_k \;\,+\;\, \tmop{Conjugate}\;,
  \end{eqnarray*}
  we infer the equalities
  \begin{eqnarray*}
    \big| \nabla^{1, 0}_{g, J} \,A\big|^2_g & = & 4\,\big| \partial_r \,A_{k, \bar{l}} \big|^2 \;\,=\;\, \big|
    \partial^g_{_{T_{X, J}}} \,A\big|_g^2 \;.
  \end{eqnarray*}
  The conclusion follows from the second variation formula (\ref{rm-2vr-W})
  for the $\mathcal{W}_{\Omega}$ functional in the Riemannian case. The last
  variation formula in the statement of theorem \ref{lm-Kah-IIvr-W} follows
  from the identity (\ref{Kah-F}).
\end{proof}

We observe that the last variation formula in lemma \ref{lm-Kah-IIvr-W}
follows also from the formula (\ref{part-vr-W}). In fact in the case $v \in
\mathbbm{F}_g$ hold the equalities
\begin{eqnarray*}
  \big| \nabla^{1, 0}_{g, J} \,A\big|^2_g \;\;=\;\; \big| \partial^g_{_{T_{X, J}}} \,A\big|_g^2 & = &
  \left \langle \partial^g_{_{T_{X, J}}} A\,, \overline{\partial}_{_{T_{X, J}}}
  B \right\rangle_g\\
  &  & \\
  & = & - \;\,g \left( \nabla^{1, 0}_{e_p} \,B\; \nabla^{1, 0}_{e_r} \,A\, e_p\,, e_r
  \right)\\
  &  & \\
  & = & g \left( \nabla^{1, 0}_{e_p} \,B\; \nabla^{0, 1}_{e_p} \,B\, e_r\,, e_r
  \right)\\
  &  & \\
  & = & \big| \nabla^{1, 0}_{g, J} \,B\big|^2_g \;\;=\;\; \big| \nabla^{0, 1}_{g, J} \,B\big|^2_g\;,
\end{eqnarray*}
thanks to the identities (\ref{sup-Kh-sm}) and (\ref{Kh-sm-B1}).

\section{Appendix}

\subsection{Operators acting on alternating and symmetric
tensors}\label{apdx1}

We need to explain in this sections a few quite elementary facts in order
to fix a convention inaccuracy which often occurs in differential geometry.
Inaccuracy which is source of frequent mistakes. 
To be precise, with our convention the metric induced on the space of forms is
the restriction of the metric on the space of tensors (without degree
multiplicative factors!).

We consider the natural projectors $A : (T_X^{\ast})^{\otimes p} \rightarrow
\Lambda^p T_X^{\ast}${\hspace{0.25em}}, $\alpha \mapsto A
(\alpha)${\hspace{0.25em}},
\begin{eqnarray*}
  A (\alpha) (v_1, ..., v_p) & = & \sum_{\sigma \in S_p} \varepsilon_{\sigma}
  \hspace{0.25em} \alpha (v_{\sigma_1}, ..., v_{\sigma_1}) \hspace{0.25em},
\end{eqnarray*}
i.e.
\begin{eqnarray*}
  A (\alpha_1 \otimes \cdots \otimes \alpha_p) & = & \bigwedge_{j = 1}^p
  \alpha_j \hspace{0.75em} = \hspace{0.75em} \sum_{\sigma \in S_p}
  \varepsilon_{\sigma} \hspace{0.25em} \alpha_{\sigma_1} \otimes \cdots
  \otimes \alpha_{\sigma_p} \hspace{0.25em},
\end{eqnarray*}
and $S : (T_X^{\ast})^{\otimes p} \rightarrow S^p
T_X^{\ast}${\hspace{0.25em}}, $\alpha \mapsto S (\alpha)${\hspace{0.25em}},
\begin{eqnarray*}
  S (\alpha) (v_1, ..., v_p) & = & \sum_{\sigma \in S_p} \alpha (v_{\sigma_1},
  ..., v_{\sigma_1}) \hspace{0.25em},
\end{eqnarray*}
i.e.
\begin{eqnarray*}
  S (\alpha_1 \otimes \cdots \otimes \alpha_p) & = & S_{j = 1}^p \alpha_j
  \hspace{0.75em} = \hspace{0.75em} \sum_{\sigma \in S_p} \alpha_{\sigma_1}
  \otimes \cdots \otimes \alpha_{\sigma_p} \hspace{0.25em} .
\end{eqnarray*}
Given a Riemannian metric $g$ over $X$ the metric induced over
$(T_X^{\ast})^{\otimes p}$ is given by
\begin{eqnarray*}
  \left\langle \alpha_1 \otimes \cdots \otimes \alpha_p, \beta_1 \otimes
  \cdots \otimes \beta_p \right\rangle_g & = & \prod_{j = 1}^p g (\alpha_j,
  \beta_j) \hspace{0.25em},
\end{eqnarray*}
The restriction of this metric to the subspaces $\Lambda^p T_X^{\ast}$ and
$S^p T_X^{\ast}$ is given respectively by the formulas
\begin{eqnarray}\label{metr-ext}
  \left\langle \alpha_1 \wedge ... \wedge \alpha_p, \beta_1 \wedge ... \wedge
  \beta_p \right\rangle_g & = & p! \det \left( g (\alpha_k, \beta_l) \right)\nonumber
\\\nonumber
  &  & \\
  & = & p! \sum_{\sigma \in S_p} \varepsilon_{\sigma} \prod_{j = 1}^p g
  (\alpha_j, \beta_{\sigma_j}) \hspace{0.25em},
\end{eqnarray}
\[  \]
\begin{eqnarray}
  \label{metr-sym}  \left\langle S_{j = 1}^p \alpha_j, S_{j = 1}^p \beta_j
  \right\rangle_g = p! \sum_{\sigma \in S_p} \prod_{j = 1}^p g (\alpha_j,
  \beta_{\sigma_j}) \hspace{0.25em} . &  & 
\end{eqnarray}
In fact let prove first the identity (\ref{metr-ext}). Let $A_{i, j} : = g
(\alpha_i, \beta_j)$. We define the components $(A_{\sigma, \tau})_{i, j}$ of
the matrix $A_{\sigma, \tau}$ as $(A_{\sigma, \tau})_{i, j} : = A_{\sigma_i,
\tau_j} = g (\alpha_{\sigma_i}, \beta_{\tau_j})$. We define also
\begin{eqnarray*}
  \prod A_{\sigma, \tau} & = & \prod_{j = 1}^p g (\alpha_{\sigma_j},
  \beta_{\tau_j}) \hspace{0.25em} .
\end{eqnarray*}
Then
\begin{eqnarray*}
  \left\langle \alpha_1 \wedge ... \wedge \alpha_p, \beta_1 \wedge ... \wedge
  \beta_p \right\rangle_g & = & \sum_{\sigma, \tau \in S_p}
  \varepsilon_{\sigma} \varepsilon_{\tau} \prod A_{\sigma, \tau}
  \hspace{0.25em} .
\end{eqnarray*}
The identity (\ref{metr-ext}) follow from
\begin{equation}
  \label{metr-ext-0}  \sum_{\sigma, \tau \in S_p} \varepsilon_{\sigma}
  \varepsilon_{\tau} \prod A_{\sigma, \tau} \hspace{0.75em} = \hspace{0.75em} p! \sum_{\sigma
  \in S_p} \varepsilon_{\sigma} \prod A_{I, \sigma}\;,
\end{equation}
that we prove below. We observe first the identity
\begin{equation}
  \label{metr-ext-1}  \sum_{\sigma \in S_p} \varepsilon_{\sigma} \prod A_{I,
  \sigma} \hspace{0.75em} = \hspace{0.75em}  \sum_{\sigma \in S_p}
  \varepsilon_{\sigma} \prod A_{\sigma, I} \;.
\end{equation}
We observe also that for any $\tau \in S_p$ hold the identity
\begin{equation}
  \label{metr-ext-2}  \sum_{\sigma \in S_p} \varepsilon_{\sigma} \prod A_{I,
  \sigma} \hspace{0.75em} = \hspace{0.75em}\varepsilon_{\tau} \sum_{\sigma \in S_p}
  \varepsilon_{\sigma} \prod A_{I, \tau \cdot \sigma}\;,
\end{equation}
since $\varepsilon_{\tau \cdot \sigma} = \varepsilon_{\sigma}
\varepsilon_{\tau}$. Moreover
\begin{eqnarray*}
  \sum_{\sigma \in S_p} \varepsilon_{\sigma} \prod A_{I, \tau \cdot \sigma} &
  = & \sum_{\sigma \in S_p} \varepsilon_{\sigma} \prod (A_{I, \tau})_{I,
  \sigma}\\
  &  & \\
  & = & \sum_{\sigma \in S_p} \varepsilon_{\sigma} \prod (A_{I,
  \tau})_{\sigma, I} \hspace{1em} \hspace{2em}  \text{\tmop{thanks} \tmop{to}
  (\ref{metr-ext-1})}\\
  &  & \\
  & = & \sum_{\sigma \in S_p} \varepsilon_{\sigma} \prod A_{\sigma, \tau}
  \hspace{0.25em} .
\end{eqnarray*}
This combined with (\ref{metr-ext-2}) implies (\ref{metr-ext-0}). The identity
(\ref{metr-sym}) follows dropping $\varepsilon_{\sigma}$ and
$\varepsilon_{\tau}$ in the previous computation.

Let now $(F, h)$ be a hermitian vector bundle over a Riemann manifold $(M, g)$
(of dimension $m$) equipped with a $h$-hermitian connection $\nabla_F$ and let
$\nabla$ be the induced hermitian connection over the hermitian vector bundle
\[ \left( (T^{\ast}_M)^{\otimes p} \otimes_{_{\mathbbm{R}}} F, \left\langle
   \cdot, \cdot \right\rangle \right) \hspace{0.25em}, \]
where $\left\langle \cdot, \cdot \right\rangle$ is the induced hermitian
product. Moreover consider the first order differential operator
\[ \nabla : C^{\infty} \left( (T^{\ast}_M)^{\otimes p}
   \otimes_{_{\mathbbm{R}}} F \right) \longrightarrow C^{\infty} \left(
   (T^{\ast}_M)^{\otimes p + 1} \otimes_{_{\mathbbm{R}}} F \right)
   \hspace{0.25em}, \]
defined by $\nabla \alpha (\xi_0, ..., \xi_p) : = \nabla_{\xi_0} \alpha
(\xi_1, ..., \xi_p)$ for all vectors $\xi_0, ..., \xi_p \in T_{M, x}$. The
$h$-hermitian connection $\nabla_F$ on $F$ extends to an exterior derivation
on the sheaf $C^{\infty} (\Lambda^p T^{\ast}_M \otimes_{_{\mathbbm{R}}} F)$
that we still denote by $\nabla_F$,
\[ \nabla_F : C^{\infty} \left( \Lambda^p T^{\ast}_M \otimes_{_{\mathbbm{R}}}
   F \right) \longrightarrow C^{\infty} \left( \Lambda^{p + 1} T^{\ast}_M
   \otimes_{_{\mathbbm{R}}} F \right) \hspace{0.25em}, \]
The relation with the operator $\nabla$ is
\begin{eqnarray*}
  \nabla_F \, \alpha \,(\xi_0, ..., \xi_p) & = & \sum_{j = 0}^p \,(-
  1)^j \, \nabla \,\alpha\, (\xi_j, \xi_0, ..., \hat{\xi}_j, ...,
  \xi_p) \;.
\end{eqnarray*}
We observe that
\[ \nabla : C^{\infty} \left( \Lambda^p T^{\ast}_M \otimes_{_{\mathbbm{R}}} F
   \right) \longrightarrow C^{\infty} \left( T^{\ast}_M
   \otimes_{_{\mathbbm{R}}} \Lambda^p T^{\ast}_M \otimes_{_{\mathbbm{R}}} F
   \right) \;. \]
In a similar way we define the operator
\[ \hat{\nabla}_F : C^{\infty} \left( S^p T^{\ast}_M \otimes_{_{\mathbbm{R}}}
   F \right) \longrightarrow C^{\infty} \left( S^{p + 1} T^{\ast}_M
   \otimes_{_{\mathbbm{R}}} F \right)\;, \]
\begin{eqnarray*}
  \hat{\nabla}_F \, \alpha\, (\xi_0, ..., \xi_p) & = & \sum_{j =
  0}^p \nabla \,\alpha\, (\xi_j, \xi_0, ..., \hat{\xi}_j, ..., \xi_p) \;.
\end{eqnarray*}
We observe in fact that
\[ \nabla : C^{\infty} \left( S^p T^{\ast}_M \otimes_{_{\mathbbm{R}}} F
   \right) \longrightarrow C^{\infty} \left( T^{\ast}_M
   \otimes_{_{\mathbbm{R}}} S^p T^{\ast}_M \otimes_{_{\mathbbm{R}}} F \right)
   \hspace{0.25em} . \]
The formal adjoint operator
\[ \nabla^{\ast} : C^{\infty} \left( (T^{\ast}_M)^{\otimes p + 1}
   \otimes_{_{\mathbbm{R}}} F \right) \longrightarrow C^{\infty} \left(
   (T^{\ast}_M)^{\otimes p} \otimes_{_{\mathbbm{R}}} F \right)
   \hspace{0.25em}, \]
of the operator $\nabla$ is given by the formula
\[ \nabla^{\ast} \alpha \hspace{0.25em} (\xi_1, ..., \xi_{p - 1})
   \hspace{0.75em} : = \hspace{0.75em} - \hspace{0.25em} \tmop{Tr}_g \nabla
   \,\alpha \,(\cdot, \cdot, \xi_1, ..., \xi_{p - 1}) . \]
We remark now that $\nabla_F^{\ast} = (p + 1) \,\nabla^{\ast}$ in restriction to
$C^{\infty} (\Lambda^{p + 1} T^{\ast}_M \otimes_{_{}} F)$, i.e
\[ \nabla_F^{\ast} \hspace{0.75em} = \hspace{0.75em} (p + 1) \,\nabla^{\ast} :
   C^{\infty} \left( \Lambda^{p + 1} T^{\ast}_M \otimes_{_{\mathbbm{R}}} F
   \right) \longrightarrow C^{\infty} \left( \Lambda^p T^{\ast}_M
   \otimes_{_{\mathbbm{R}}} F \right) \hspace{0.25em} . \]
In fact this follows from the identity
\begin{equation}
  \label{cntr-altfm}  \left\langle \nabla_F \hspace{0.25em} \alpha\,, \beta
  \right\rangle = (p + 1) \left\langle \nabla \alpha\,, \beta \right\rangle\;,
\end{equation}
for any $F$-valued $p$-form $\alpha$ and any $F$-valued $(p + 1)$-form $\beta$.
Moreover we observe that $\hat{\nabla}_F^{\ast} = (p + 1) \nabla^{\ast}$ in restriction to
$C^{\infty} (S^{p + 1} T^{\ast}_M \otimes_{_{}} F)$, i.e
\[ \hat{\nabla}_F^{\ast} \hspace{0.75em} = \hspace{0.75em} (p + 1)
   \nabla^{\ast} : C^{\infty} \left( S^{p + 1} T^{\ast}_M
   \otimes_{_{\mathbbm{R}}} F \right) \longrightarrow C^{\infty} \left( S^p
   T^{\ast}_M \otimes_{_{\mathbbm{R}}} F \right) \hspace{0.25em} . \]
In fact this follows from the identity
\begin{equation}
  \label{cntr-smfm}  \left\langle \hat{\nabla}_F \hspace{0.25em} \alpha\,, \beta
  \right\rangle = (p + 1) \left\langle \nabla \alpha\,, \beta \right\rangle\;,
\end{equation}
for any $F$-valued symmetric $p$-tensor $\alpha$ and any $F$-valued symmetric $(p
+ 1)$-tensor $\beta$. 
Let prove now the identities (\ref{cntr-altfm}) and
(\ref{cntr-smfm}).

{\tmstrong{Proof of the identity (\ref{cntr-altfm})}}. Let $(\theta_s)_s$ be
a $h$-orthonormal frame of the bundle $F$ of complex rank $r$ and let
$(e_k)_k$ be a $g$-orthonormal frame of $T_X$. For any $I = (i_1, ..., i_p)$,
$1 \le i_k < i_{k + 1} \le n$, we define $\underline{e}_I^{\ast} \assign
e_{i_1} \wedge \ldots \wedge e_{i_p}$. 
We observe that $(\underline{e}_I^{\ast})_{|I| = p}$ is a local frame of the bundle $\Lambda^p
T^{\ast}_M$ which satisfies $\left\langle \underline{e}_I^{\ast},
\underline{e}_J^{\ast} \right\rangle_g = 0$ if and only if $I {\not =} J$. Moreover $|
\underline{e}_I^{\ast} |^2_g = p!$. We observe also that the coefficients of
the local expressions
\begin{eqnarray*}
\nabla \alpha 
&=&
\sum_{s = 1}^r \,\sum_{j =
   1}^m \,\sum_{|I| = p} \,C_{j, I}^s \hspace{0.25em} e_j^{\ast} \otimes
   \underline{e}^{\ast}_I \otimes \theta_s \;,
\\
\\
\nabla_F\, \alpha
&=&
\sum_{s = 1}^r
   \sum_{|K| = p + 1} B_K^s \hspace{0.25em} \underline{e}^{\ast}_K \otimes \theta_s\;,
\end{eqnarray*}
are related by the formula
\begin{eqnarray*}
  B_K^s & = & \sum_{j = 0}^p (- 1)^j C^s_{k_j, \hat{K}_j}\;,
\end{eqnarray*}
where $\hat{K}_j : = (k_0, ..., \hat{k}_j, ..., k_p)$. On the other hand
\begin{eqnarray*}
  (p + 1) \left\langle \nabla \alpha\,, \underline{e}^{\ast}_K \otimes \theta_s
  \right\rangle & = & (p + 1) \sum_{j = 1}^n \left\langle \nabla_{e_j} \alpha\,,
  e_j \;\neg\; \underline{e}^{\ast}_K \otimes \theta_s \right\rangle \\
  &  & \\
  & = & (p + 1) \sum_{j = 1}^p \left\langle \nabla_{e_{k_j}} \alpha\,, e_{k_j}
  \;\neg\; \underline{e}^{\ast}_K \otimes \theta_s \right\rangle\\
  &  & \\
  & = & (p + 1) \sum_{j = 1}^p (-1)^j\left\langle \nabla_{e_{k_j}} \alpha\,,
  \underline{e}^{\ast}_{\hat{K}_j} \otimes \theta_s \right\rangle \\
  &  & \\
  & = & (p + 1)! \sum_{j = 0}^p (-1)^jC^s_{k_j, \hat{K}_j} \\
  &  & \\
  & = & (p + 1)! \,B_K^s\\
  &  & \\
  & = & \left\langle \nabla_F \hspace{0.25em} \alpha\,, \underline{e}^{\ast}_K
  \otimes \theta_s \right\rangle \hspace{0.25em},
\end{eqnarray*}
which proves the required identity (\ref{cntr-altfm}).

{\tmstrong{Proof of the identity (\ref{cntr-smfm})}}.
For any $I = (i_1, ..., i_p)$, $1 \le i_k
\le i_{k + 1} \le n$, we define
\[ e^{\ast}_I \hspace{0.75em} : = \hspace{0.75em} \sum_{\sigma \in S_p}
   e^{\ast}_{i_{\sigma_1}} \otimes \cdots \otimes e^{\ast}_{i_{\sigma_p}}
   \hspace{0.25em}. \]
We observe that $(e^{\ast}_I)_{|I| = p}$ is a local frame of the bundle $S^p T^{\ast}_M$ which
satisfies $\left\langle e_I^{\ast}, e_J^{\ast} \right\rangle_g = 0$ if and only if $I
{\not =} J$. Moreover
\begin{eqnarray*}
|e_I^{\ast} |^2_g & = & p ! \,e_I^{\ast} (e_{i_1}, \ldots, e_{i_p}) \;.
\end{eqnarray*}
We observe also that the coefficients of the local expressions
\begin{eqnarray*}
\nabla \alpha 
&=&
\sum_{s = 1}^r \,\sum_{j =
   1}^m \,\sum_{|I| = p} \,C_{j, I}^s \hspace{0.25em} e_j^{\ast} \otimes
   e^{\ast}_I \otimes \theta_s \;,
\\
\\
\hat{\nabla}_F\, \alpha
&=&
\sum_{s = 1}^r
   \sum_{|K| = p + 1} B_K^s \hspace{0.25em} e^{\ast}_K \otimes \theta_s\;,
\end{eqnarray*}
are related by the formula
\begin{eqnarray*}
  B_K^s \,|e_K^{\ast} |^2_g & = & (p + 1) \sum_{j = 0}^p C^s_{k_j, \hat{K}_j}
  \, |e^{\ast}_{\hat{K}_j} |^2_g\;,
\end{eqnarray*}
where $\hat{K}_j : = (k_0, ..., \hat{k}_j, ..., k_p)$. On the other hand
\begin{eqnarray*}
  (p + 1) \left\langle \nabla \alpha\,, e^{\ast}_K \otimes \theta_s
  \right\rangle & = & (p + 1) \sum_{j = 1}^n \left\langle \nabla_{e_j} \alpha\,,
  e_j \;\neg\; e^{\ast}_K \otimes \theta_s \right\rangle \\
  &  & \\
  & = & (p + 1) \sum_{j = 1}^p N^{-1}_j\left\langle \nabla_{e_{k_j}} \alpha\,, e_{k_j}
  \neg\; e^{\ast}_K \otimes \theta_s \right\rangle\;,
\end{eqnarray*}
where $N_j \assign \tmop{Card} \left\{ r \in \left\{ 0, \ldots p\} \mid k_r =
k_j \right\} \geqslant 1 \right.$. Furthermore
\begin{eqnarray*}
  e_{k_j} \neg \;e^{\ast}_K & = & \sum_{\sigma \in S_{p + 1}}
  e^{\ast}_{k_{\sigma_0}} \otimes \cdots (e^{\ast}_{k_{\sigma_j}} \cdot
  e_{k_j}) \cdots \otimes e^{\ast}_{k_{\sigma_p}}\\
  &  & \\
  & = & \sum_{\sigma \in S_{p + 1}, k_{\sigma_j} = k_j}
  e^{\ast}_{k_{\sigma_0}} \otimes \cdots \otimes
  \widehat{e^{\ast}_{k_{\sigma_j}}} \otimes \cdots \otimes
  e^{\ast}_{k_{\sigma_p}}\\
  &  & \\
  & = & N_j \sum_{\sigma \in S_{p + 1}, \sigma_j = j} e^{\ast}_{k_{\sigma_0}}
  \otimes \cdots \otimes \widehat{e^{\ast}_{k_{\sigma_j}}} \otimes \cdots
  \otimes e^{\ast}_{k_{\sigma_p}}\\
  &  & \\
  & = & N_j\, e^{\ast}_{\hat{K}_j} \;.
\end{eqnarray*}
We infer the identities
\begin{eqnarray*}
  (p + 1) \left\langle \nabla \alpha\,, e^{\ast}_K \otimes \theta_s
  \right\rangle 
  & = & (p + 1) \sum_{j = 1}^p \left\langle \nabla_{e_{k_j}} \alpha\,,
  e^{\ast}_{\hat{K}_j} \otimes \theta_s \right\rangle \\
  &  & \\
  & = & (p + 1) \sum_{j = 0}^p C^s_{k_j, \hat{K}_j} |e^{\ast}_{\hat{K}_j}
  |^2_g\\
  &  & \\
  & = & B_K^s \,|e_K^{\ast} |^2_g\\
  &  & \\
  & = & \left\langle \hat{\nabla}_F \hspace{0.25em} \alpha\,, e^{\ast}_K
  \otimes \theta_s \right\rangle \hspace{0.25em},
\end{eqnarray*}
which show the required identity (\ref{cntr-smfm}).

\subsection{The first variation of Perelman's $\mathcal{W}$
functional}\label{apdx2}

We consider Perelman's $\mathcal{W}$-functional
\begin{eqnarray*}
  \mathcal{W} (g, f) & : = & \int_X \left( | \nabla_g f|^2_g \hspace{0.75em} +
  \hspace{0.75em} \tmop{Scal}_g \hspace{0.75em} + \hspace{0.75em} 2 f
  \hspace{0.75em} - m \right) e^{- f} dV_g\\
  &  & \\
  & = & \int_X \left( \Delta_g f \hspace{0.75em} + \hspace{0.75em}
  \tmop{Scal}_g \hspace{0.75em} + \hspace{0.75em} 2 f \hspace{0.75em} - m
  \right) e^{- f} dV_g \hspace{0.25em} .
\end{eqnarray*}
(We use here the identity $\Delta_g e^{- f} = (| \nabla_g f|^2_g - \Delta_g f)
e^{- f}$.) If we set
\[ \mathcal{W}_{\Omega} (g) \hspace{0.75em} : = \hspace{0.75em} \mathcal{W}
   (g, \log (dV_g / \Omega)) \hspace{0.25em}, \hspace{2em} \text{\tmop{and}}
   \hspace{2em} h_g^{\Omega} \hspace{0.75em} : = \hspace{0.75em} \tmop{Ric}_g
   (\Omega) \hspace{0.75em} - \hspace{0.75em} g \hspace{0.25em}, \]
then hold the identity
\[ \mathcal{W}_{\Omega} (g) \hspace{0.75em} = \hspace{0.75em} \int_X \left(
   \tmop{Tr}_g h_g^{\Omega} \hspace{0.75em} + \hspace{0.75em} 2 \,\log
   \frac{dV_g}{\Omega} \right) \Omega \hspace{0.25em}, \]
Thus if $(g_t)_t$ is a family of Riemannian metrics and $h_t : =
h_{g_t}^{\Omega}$ then
\begin{eqnarray*}
  \frac{d}{dt} \hspace{0.25em} \mathcal{W}_{\Omega} (g_t) & = & \frac{d}{dt}
  \hspace{0.25em} \int_X \left[ \tmop{Tr}_{_{\mathbbm{R}}} (g_t^{- 1} h_t)
  \hspace{0.75em} + \hspace{0.75em} 2 \log \frac{dV_{g_t}}{\Omega} \right]
  \Omega\\
  &  & \\
  & = & \int_X \left[ \tmop{Tr}_{_{\mathbbm{R}}} \left( - \hspace{0.25em}
  \dot{g}_t^{\ast} h^{\ast}_t \hspace{0.75em} + \hspace{0.75em}
  \dot{h}^{\ast}_t \right) \hspace{0.75em} + \hspace{0.75em} \tmop{Tr}_{g_t}
  \dot{g} \right] \Omega\\
  &  & \\
  & = & \int_X \big\langle \dot{g}_t\,, - \hspace{0.25em} h_t
  \big\rangle_{g_t} \Omega \hspace{0.75em} + \hspace{0.75em} \int_X
  \left\langle g_t\,, \frac{d}{dt} \hspace{0.25em} \tmop{Ric}_{g_t} (\Omega)
  \right\rangle_{g_t} \Omega \hspace{0.25em} .
\end{eqnarray*}
Using corollary \ref{vr-ORc-adj} and integrating by parts we infer
\begin{eqnarray*}
  \hspace{0.25em} \int_X \left\langle g_t\,, \frac{d}{dt} \hspace{0.25em}
  \tmop{Ric}_{g_t} (\Omega) \right\rangle_{g_t} \Omega & = & - \hspace{0.75em}
  \frac{1}{6} \int_X \left\langle g_t\,, \hat{\nabla}_{g_t}^{\ast_{_{\Omega}}} 
  \hspace{0.25em} \hat{\nabla}_{g_t}  \dot{g}_t \right\rangle_{g_t} \Omega\\
  &  & \\
  & + & \hspace{0.25em} \int_X \left\langle g_t\,,
  \nabla_{g_t}^{\ast_{_{\Omega}}} \nabla_{g_t}  \dot{g}_t \right\rangle_{g_t}
  \Omega\\
  &  & \\
  & = & - \hspace{0.75em} \frac{1}{6} \,\int_X \left\langle \hat{\nabla}_{g_t}
  g_t\,, \hat{\nabla}_{g_t}  \dot{g}_t \right\rangle_{g_t} \Omega\\
  &  & \\
  & + & \int_X \big\langle \nabla_{g_t} g_t\,, \nabla_{g_t}  \dot{g}_t
  \big\rangle_{g_t} \Omega\\
  &  & \\
  & = & 0 \hspace{0.25em},
\end{eqnarray*}
which implies Perelman's \cite{Per} variation formula
\begin{eqnarray*}
  \frac{d}{dt} \hspace{0.25em} \mathcal{W}_{\Omega} (g_t) & = & - \;\int_X
  \big\langle \dot{g}_t, h_t \big\rangle_{g_t} \Omega\; .
\end{eqnarray*}
\subsection{The curvature of the space $(\mathcal{M}, G)$}\label{apdx3}
We show that the Riemannian space $(\mathcal{M}, G)$ is non-positively
curved. In fact let $u, v, w \in \mathcal{H}$ and consider them also like
constant $\mathcal{H}$-valued vector fields over $\mathcal{M}.$ Using the expression (\ref{Krist-G}) we expand the
Riemannian curvature tensor $\mathcal{R}_G$ as follows;
\begin{eqnarray*}
  2\,\mathcal{R}_G (u, v) w & = & 2\, \nabla_{G, u} \nabla_{G, v} w \;\,-\;\, 2\, \nabla_{G,
  v} \nabla_{G, u} w\\
  &  & \\
  & = & 2 \,\nabla_{G, u}  \left( \Gamma_G (v, w) \right) \;\,-\;\, 2\, \nabla_{G, v} 
  \left( \Gamma_G (u, w) \right)\\
  &  & \\
  & = & 2 \,D \left( \Gamma_G (v, w) \right) (u) \;\,-\;\, 2\, D \left( \Gamma_G (u, w)
  \right) (v)\\
  &  & \\
  & + & 2 \,\Gamma_G (u, \Gamma_G (v, w)) \;\,-\;\, 2\, \Gamma_G (v, \Gamma_G (u, w))
\\
\\
  & = & v\, u_g^{\ast} w_g^{\ast} \;\,+\;\, w\, u_g^{\ast} v_g^{\ast} \;\,-\;\, u \,\Gamma_G (v,
  w)_g^{\ast} \;\,-\;\, \Gamma_G (v, w) \,u_g^{\ast} \\
  &  & \\
  & - & u \,v_g^{\ast} w_g^{\ast} \;\,-\;\, w\, v_g^{\ast} u_g^{\ast} \;\,+\;\, v \,\Gamma_G (u,
  w)_g^{\ast} \;\,+\;\, \Gamma_G (u, w) \,v_g^{\ast}\\
  &  & \\
  & = & - \;\,g \left[ u_g^{\ast} \hspace{0.25em}, v^{\ast}_g \right] w^{\ast}_g
  \;\,+\;\, w \left[ u_g^{\ast} \hspace{0.25em}, v^{\ast}_g \right]\\
  &  & \\
  & + & \frac{1}{2}\; u \,(v_g^{\ast} \,w_g^{\ast} \;\,+\;\, w^{\ast}_g \,v_g^{\ast}) \;\,+\;\,
  \frac{1}{2}\; (v\, w_g^{\ast} \;\,+\;\, w \,v_g^{\ast})\, u^{\ast}_g\\
  &  & \\
  & - & \frac{1}{2} \;v\, (u_g^{\ast} w_g^{\ast} \;\,+\;\, w^{\ast}_g u_g^{\ast}) \;\,-\;\,
  \frac{1}{2} (u\, w_g^{\ast} \;\,+\;\, w\, u_g^{\ast}) \,v^{\ast}_g \\
  &  & \\
  & = &  - \;\,\frac{1}{2} \;g \Big[ \left[ u_g^{\ast} \hspace{0.25em},
  v^{\ast}_g \right], w^{\ast}_g \Big] \;.
\end{eqnarray*}
We obtain the well known formula for the curvature operator of the metric $G$,
\begin{equation}
  \label{curvat-formula} \mathcal{R}_G (g) (u, v) w\;\; =\;\; -\;\, \frac{1}{4}\; g \Big[
  \left[ u_g^{\ast} \hspace{0.25em}, v^{\ast}_g \right], w^{\ast}_g \Big]\; .
\end{equation}
We remind that the curvature form is defined by the identity
$$
R_G (u, v, \xi, \eta) \;\;\assign\;\;
G (\mathcal{R}_G (u, v) \eta, \xi)\;,
$$ 
and the sectional curvature is given by
the formula
\begin{eqnarray*}
  \sigma_G (u, v) & \assign &  R_G (u, v, u, v) \;\;=\;\; -\;\, \frac{1}{4}\,
  \int_X \tmop{Tr}_{_{\mathbbm{R}}} \Big\{ \Big[ \left[ u_g^{\ast} \hspace{0.25em},
  v^{\ast}_g \right], v^{\ast}_g \Big] u^{\ast}_g \Big\}\; \Omega\; .
\end{eqnarray*}
We set for notation simplicity $U : = u^{\ast}_g$ and $V \assign v^{\ast}_g$ .
We expand the expression
\begin{eqnarray*}
  \Big[ \left[ U, V \right], V \Big] U & = & \left( \left[ U, V \right] V
  \;\,-\;\, V \left[ U, V \right]  \right) U\\
  &  & \\
  & = & \left( U V^2 \;\,-\;\, 2\, V U V \;\,+\;\, V^2 U \right) U\\
  &  & \\
  & = & U V^2 U \;\,+\;\, V^2 U^2 \;\,-\;\, 2\, (V U)^2\; .
\end{eqnarray*}
Taking the trace we obtain
\begin{eqnarray*}
  \tmop{Tr}_{_{\mathbbm{R}}} \Big\{ \Big[ \left[ U, V \right], V \Big] U
  \Big\} & = & 2 \,\tmop{Tr}_{_{\mathbbm{R}}} \left[ V^2 U^2 \;\,-\;\, (V U)^2 \right]
  \;.
\end{eqnarray*}
On the other hand we expand the therm
\begin{eqnarray*}
  \left[ U, V \right]^2 & = & (U V)^2 \;\,-\;\, V U^2 V \;\,-\;\, U V^2 U \;\,+\;\, (V U)^2 \;.
\end{eqnarray*}
Taking the trace we deduce
\begin{eqnarray*}
  \tmop{Tr}_{_{\mathbbm{R}}}  \left[ U, V \right]^2 & = &
  \tmop{Tr}_{_{\mathbbm{R}}} \left[ (U V)_g^T (U V)_g^T \;\,-\;\, U^2 V^2 \;\,-\;\, V^2 U^2 \;\,+\;\,
  (V U)^2 \right]\\
  &  & \\
  & = & 2 \,\tmop{Tr}_{_{\mathbbm{R}}} \left[ (V U)^2 \;\,-\;\, V^2 U^2 \right]\; .
\end{eqnarray*}
Thus
\begin{eqnarray*}
  \tmop{Tr}_{_{\mathbbm{R}}} \Big\{ \Big[ \left[ U, V \right], V \Big] U
  \Big\} & = & - \tmop{Tr}_{_{\mathbbm{R}}}  \left[ U, V \right]^2 
\\
\\
&=&
  \tmop{Tr}_{_{\mathbbm{R}}} \left\{ \left[ U, V \right] \cdot \left[ U, V
  \right]_g^T  \right\}\;,
\end{eqnarray*}
since the endomorphism $[U, V]$ is $g$-anti-symmetric due to the fact that $U,
V$ are $g$-symmetric. We infer the well known formula for the sectional
curvature of the metric $G$,
\[ \sigma_G (u, v) \;\;=\;\; -\;\, \frac{1}{4} \,\int_X \big| \left[ u_g^{\ast} \hspace{0.25em},
   v^{\ast}_g \right] \big|^2_g \;\Omega \;\;\leqslant\;\; 0 \;. \]
This shows the required conclusion.
\newpage
{\bf Acknowledgments}. We warmly thanks Patrick G\'erard and Pierre Pansu for
stimulating discussions. We thanks in particular Pierre Pansu for pointing
out an inaccuracy in our computation of the classic curvature formula
(\ref{curvat-formula}) in a preliminary version of the manuscript.

\vspace{1cm}
\noindent
Nefton Pali
\\
Universit\'{e} Paris Sud, D\'epartement de Math\'ematiques 
\\
B\^{a}timent 425 F91405 Orsay, France
\\
E-mail: \textit{nefton.pali@math.u-psud.fr}
\end{document}